\documentclass[10pt,a4paper]{article}
\usepackage{amsmath}
\usepackage[utopia]{mathdesign}

\usepackage{cite}
\usepackage[left=2.5cm,right=2.5cm,top=2.5cm,bottom=2.5cm]{geometry}
\usepackage{graphicx,xcolor}
\def\R{\mathbb{R}}
\def\N{\mathbb{N}}

\usepackage{enumerate} 
\usepackage[shortlabels]{enumitem} 
\numberwithin{equation}{section}

\usepackage[thref,thmmarks,standard,amsmath,hyperref]{ntheorem}
\usepackage[unicode]{hyperref}
\theoremheaderfont{\normalfont\bfseries}
\theorembodyfont{\normalfont}
\theoremseparator{.}

\theorembodyfont{\itshape}
\renewtheorem{theorem}{Theorem}[section]
\renewtheorem{proposition}{Proposition}[section]
\renewtheorem{lemma}{Lemma}[section]
\renewtheorem{corollary}{Corollary}[section]

\theorembodyfont{\normalfont}
\renewtheorem{definition}{Definition}[section]
\renewtheorem{remark}{Remark}[section]
\renewtheorem{example}{Example}[section]

\theoremstyle{plain}
\theoremstyle{nonumberplain}
\theoremsymbol{\ensuremath{\square}}
\renewtheorem{proof}{Proof}

\makeatletter
\let\c@proposition\c@theorem

\let\c@lemma\c@theorem

\let\c@corollary\c@theorem

\let\c@definition\c@theorem

\let\c@remark\c@theorem

\let\c@example\c@theorem

\makeatother

\title{\textbf{Convex Generalized Differentiation at infinity}}
\author{Nguyen Xuan Duy Bao\thanks{Independent Researcher, Ho Chi Minh City, Vietnam (nxdbao@gmail.com)} ~, 
Nguyen Mau Nam\thanks{Fariborz Maseeh Department of Mathematics and Statistics, Portland State University, Portland, OR, USA (mnn3@pdx.edu)}}
\date{\today}

\newcommand{\cl}{{\rm cl}}
\newcommand{\bd}{{\rm bd}}

\newcommand{\dom}{{\rm dom}}
\newcommand{\rge}{{\rm rge}}
\newcommand{\gph}{{\rm gph}}
\newcommand{\cone}{{\rm cone}}
\newcommand{\inte}{{\rm int}}
\newcommand{\conv}{{\rm conv}}
\newcommand{\epi}{{\rm epi}}

\newcommand {\Limsup} {\mathop{{\rm Lim\,sup}\,}}
\newcommand {\Liminf} {\mathop{{\rm Lim\,inf}\,}}

\begin{document}
\maketitle
\thispagestyle{empty}

\begin{abstract}
In this paper, we develop a generalized differentiation theory at infinity for convex sets and functions. In particular, we study several fundamental notions of convex analysis at infinity, including tangent cones, normal cones, and subdifferentials. Our work complements the recently developed theories of nonsmooth analysis at infinity by focusing on the convex setting, where these constructions preserve convexity and admit natural connections with recession analysis, polarity, and epigraphical geometry. We show that, in the convex case, many results can be established under weaker assumptions and admit simpler and more explicit representations than those available in the general nonsmooth framework. In addition, we develop calculus rules and geometric characterizations for these constructions and apply them to optimality conditions and attainment criteria for convex optimization problems over unbounded feasible sets. The results obtained in this paper provide new tools for the study of convex sets and functions at infinity and further strengthen the connections between convex analysis, variational analysis, and optimization.

\end{abstract}



\section{Introduction}
Central concepts of nonsmooth analysis such as tangent cones, normal cones, and subdifferentials are traditionally formulated at reference points in the underlying space; see, e.g., \cite{rockafellar-convex,aubin,nambook,rockafellar-wets}. However, in many optimization and variational problems, the relevant first-order behavior is asymptotic rather than local. This occurs, for instance, when a function is bounded below but fails to attain its infimum, or when the geometry of a feasible set is determined by directions of escape and asymptotic tangency. In such situations, classical tools of nonsmooth analysis at finite points are no longer sufficient. This  motivativates the development of a theory of generalized differentiation at infinity.

The asymptotic viewpoint has been known in convex and nonsmooth analysis. For example, recession cones and asymptotic cones have been widely used to study the existence of minimizers, boundedness of solution sets, and coercivity in noncompact optimization problems; see, e.g., \cite{rockafellar-convex,aubin,nambook,rockafellar-wets}. More recently, researchers have begun to develop a systematic theory of variational analysis at infinity. In \cite{tung23a}, Nguyen and Pham introduced Clarke-type tangent cones, normal cones, subgradients, optimality conditions, and Lipschitz properties at infinity. Later, Kim, Nguyen, and Pham \cite{tung23b} developed a Mordukhovich-type framework based on normal cones and limiting and singular subdifferentials at infinity, together with calculus rules and applications to optimization. These ideas have since been extended to minimax programming, coderivatives of set-valued mappings, and directional and relative constructions at infinity; see, for example, \cite{tuyen-bae-kim-minimax,kim-pham-tung-tuyen-coderivative,kien-tuyen-nghi-directional,anh-hung-relative,anh-hung-jgo}.

These theories were developed for general nonsmooth settings and therefore do not fully exploit the additional structure available in convex analysis. In particular, even for convex sets, limiting normal cones at infinity may fail to be convex (Example~\ref{ex: nonconvex normal cone}, Remark~\ref{rem: limiting normal nonconvex}), which limits their compatibility with polarity, duality, and epigraphical geometry. Moreover, several results in the convex setting can be established under weaker assumptions and admit simpler representations.

The present paper studies tangent cones, normal cones, and subdifferentials at infinity in the setting of convex sets and functions. Our focus is on the convex case, where these objects remain convex and possess natural connections with polarity and recession constructions. This additional structure leads to more transparent formulas and simplified proofs; see Proposition~\ref{prop: cones at infinity}, Proposition~\ref{prop: basic properties}, Remark~\ref{rem: recession strict}, and Examples~\ref{ex: nonconvex tangent reverse} and~\ref{ex: linear-image-nonconvex}.

The main contributions of the paper are as follows.

\begin{itemize}
\item For convex sets, we characterize tangent and normal cones at infinity in terms of their finite-point counterparts (Proposition~\ref{prop: cones at infinity}, Proposition~\ref{prop: basic properties}) and relate them to recession geometry. In particular, Proposition~\ref{prop: recession inclusion} establishes the inclusion of recession directions in tangent directions at infinity, while Remark~\ref{rem: recession strict} shows that the inclusion may be strict.

\item For polyhedral convex sets, we obtain an explicit description of tangent and normal cones at infinity in terms of asymptotically active constraints (Theorem~\ref{prop: polyhedral cones at infinity}), together with a verifiable criterion for the equality $T(\infty_I;\Omega)=\Omega^\infty$ (Remark~\ref{rem: polyhedral recession equality}). We also establish product, intersection, and linear-image rules (Proposition~\ref{prop: normal product rules}, Proposition~\ref{prop: normal intersection rule}, Proposition~\ref{prop: normal linear image}), along with polyhedral refinements. Examples~\ref{ex: nonconvex tangent reverse} and~\ref{ex: linear-image-nonconvex} show that some of these results are specific to the convex setting.

\item For convex functions, we introduce subdifferentials at infinity through the asymptotic normal geometry of the epigraph (Definition~\ref{def: subdiff infinity}). We show that the subdifferential at infinity is a closed convex set and that the singular subdifferential at infinity is a closed convex cone (Proposition~\ref{prop: basic subdiff at infinity}), providing a convex counterpart to the limiting theory developed in \cite{tung23b}.

\item Using the Fenchel conjugate $f^*$ and the recession function $f^\infty$, we further characterize these constructions. Proposition~\ref{prop: geometry of singular} establishes the inclusions
\[
    N(\infty;\dom f)\subset\partial^\infty f(\infty)\subset (\dom f^\infty)^\circ,
\]
while Theorem~\ref{prop: geometry of subdifferential} shows that
\[
    \partial f(\infty)\subset \cl(\dom f^*),
\]
with equality whenever $f$ is sublinear. We also derive sum, chain, maximum, infimal-convolution, and marginal-function rules (Propositions~\ref{prop: sum rule same-variable}--\ref{prop: marginal at infinity}).

\item We introduce a Clarke-type directional derivative at infinity and establish its connections with the subdifferential, the singular subdifferential, and the recession function (Proposition~\ref{prop: clarke deriv epi}, Proposition~\ref{prop: clarke deriv support}, Proposition~\ref{prop: clarke deriv vs recession}).

\item As an application, we develop a multiplier rule at infinity (Theorem~\ref{thm: multiplier rule at infinity}) and an attainment criterion based on the absence of asymptotic stationarity (Theorem~\ref{thm: attainment infinity}) for convex programs over unbounded feasible sets. These results are supported by a verifiable asymptotic Slater-type condition (Proposition~\ref{thm: asymptotic Slater}) and a KKT-type normal-cone decomposition (Corollary~\ref{cor: KKT at infinity}).

\end{itemize}

The paper is organized as follows. Section~2 reviews the preliminaries used throughout the paper. Section~3 studies tangent and normal cones at infinity for convex sets, emphasizing their connections with finite-point constructions, polarity, recession geometry, and calculus rules. Section~4 introduces subdifferentials at infinity for convex functions, develops their basic properties and characterizations via the Fenchel conjugate and the recession function, and establishes calculus rules together with directional-derivative formulas. Section~5 applies these results to multiplier rules and attainment criteria for convex programs over unbounded feasible sets under an asymptotic Slater-type condition. The final section contains concluding remarks and directions for future research.

\section{Preliminaries}
Throughout the paper, \(\R^n\) is equipped with the usual inner product \(\langle x,y\rangle:=\sum_{i=1}^n x_i y_i\) and the associated Euclidean norm \(\|x\|:=\sqrt{\langle x,x\rangle}\). For $x\in\R^n$ and $r>0$, the open ball of radius $r$ centered at $x$ is denoted by $B(x;r)$. For a set $K\subset\R^n$, we denote by $\cl K$ its topological closure, by $\inte K$ its topological interior, and by $\cone K:=\{\lambda x\mid \lambda\ge 0,\ x\in K\}$ its conic hull. The set $K$ is called a nonempty cone if \(\lambda x\in K\) for all \(x\in K\) and all \(\lambda\ge 0\). In this section we collect a few notions and basic facts that will be used repeatedly later: projections and distance to convex sets, polarity for convex cones, and several standard constructions from variational analysis. Given a nonempty set $C \subset \R^n$, define the \emph{distance to $C$}, the \emph{indicator function}, and the \emph{support function} of $C$, respectively, by
    \[
    {\rm d}(w; C) := \inf_{y\in C}\|w-y\|,\quad
    \delta_C(x):=
    \begin{cases}
    0,&x\in C,\\
    \infty,&x\notin C,
    \end{cases}
    \quad
    \sigma_C(x):=\sup_{u\in C}\langle u,x\rangle.
    \]
If $C$ is nonempty, closed, and convex, then for every $w \in \R^n$ there exists a unique projection point $P_C(w)\in C$ such that
    \[
    {\rm d}(w; C)=\|w-P_C(w)\|.
    \]
Next, we recall the polar construction. The \emph{polar} of $C$ is
    \[
    C^{\circ} := \{v \in \R^n \mid \langle v, x \rangle \le 1 \;\ \text{for all} ~ x \in C\}.
    \]
The set $C^{\circ}$ is nonempty, closed, and convex. The polar of $C^{\circ}$, denoted by $C^{\circ\circ}$, is called the \emph{bipolar} of $C$. When $C$ is a cone, the polar has the following representation:
\[
C^{\circ} = \{v \in \R^n \mid \langle v, x \rangle \le 0 \;\ \text{for all} ~ x \in C\}.
\]
We recall some standard polarity rules. The \emph{bipolar theorem} for a closed convex cone, asserting $C^{\circ\circ}=C$, is given in \cite[Proposition~5.23]{nambook}. The monotonicity and product-to-intersection rules below appear as \cite[Lemma~1.24(a),(b)]{jahnbook}, and the dual identity $(C_1\cap C_2)^\circ=\cl(C_1^\circ+C_2^\circ)$ follows by taking polars and invoking the bipolar theorem on both sides.

\begin{lemma} \label{lem: polarity rules}
Let $C_1,C_2 \subset \R^n$ be nonempty convex cones, and let $C \subset \R^n$ be a nonempty closed convex cone. Then:
\begin{enumerate}[\rm (i)]
    \item $C^{\circ\circ}=C$.
    \item $C_1\subset C_2 \ \Longrightarrow\ C_2^\circ\subset C_1^\circ$.
    \item
    \(
    (C_1+C_2)^\circ=C_1^\circ\cap C_2^\circ
    \)
    \item 
    \( (C_1\cap C_2)^\circ=\cl\big(C_1^\circ+C_2^\circ\big).
    \)
\end{enumerate}
\end{lemma}

We will also use a basic relation between a cone and its polar given by the Moreau decomposition~\cite{moreau1962}. For a nonempty closed convex cone $C \subset \R^n$ and any $w\in\R^n$,
    \begin{equation} \label{eq: moreau}
    w = P_C(w) + P_{C^\circ}(w), \quad \langle P_C(w),\,P_{C^\circ}(w)\rangle = 0,
    \end{equation}
and
    \begin{equation} \label{eq: polar distance}
    {\rm d}(w; C)= \|w-P_C(w)\| =\|P_{C^\circ}(w)\|.
    \end{equation}
The next lemma is an immediate consequence of \eqref{eq: moreau} and \eqref{eq: polar distance}; see also \cite{nambook}.

\begin{lemma} \label{lemma: moreau}
Let $C \subset \R^n$ be a nonempty closed convex cone. Then
\begin{align}
    &\langle u,w\rangle \le \|u\|\,{\rm d}(w; C) \;\ \text{for all} ~ u\in C^\circ ~ \text{and} ~ w\in\R^n; \label{eq: polar bound} \\[6pt]
    &{\rm d}(w; C) = \sup\{\langle s,w\rangle \mid s\in C^\circ,\ \|s\|\le 1\}. \label{eq: distance supremum}
\end{align}
\end{lemma}

\begin{proof}
Fix any $u\in C^\circ$ and $w \in \R^n$. By \eqref{eq: moreau} we have $w=P_C(w)+P_{C^\circ}(w)$. Since $u\in C^\circ$ and $P_C(w)\in C$, we see that $\langle u,P_C(w)\rangle\le 0$. Hence
    \[
    \langle u,w\rangle
    =\langle u,P_C(w)\rangle+\langle u,P_{C^\circ}(w)\rangle
    \le \langle u,P_{C^\circ}(w)\rangle
    \le \|u\|\,\|P_{C^\circ}(w)\|.
    \]
Then \eqref{eq: polar bound} follows directly from \eqref{eq: polar distance}. 

Set $p := P_{C^\circ}(w)$ so that ${\rm d}(w; C)=\|p\|$ by \eqref{eq: polar distance}. Fix any $s \in C^\circ$ such that $\|s\|\le 1$. Then
    \[
    \langle s,w\rangle
    =\langle s,P_C(w)\rangle+\langle s,p\rangle
    \le \langle s,p\rangle
    \le \|s\|\,\|p\|
    \le \|p\|.
    \]
Hence $\sup\{\langle s,w\rangle \mid \ s\in C^\circ,\ \|s\|\le 1\}\le \|p\| = {\rm d}(w; C)$. If $p \neq 0$, choose $s=p/\|p\|$. Then by \eqref{eq: moreau} we have 
    \[
    \langle s,w\rangle
    =\dfrac{1}{\|p\|}\langle P_{C^\circ}(w),P_C(w)\rangle+\dfrac{1}{\|p\|}\langle p,p\rangle
    = \|p\|.
    \]
If $p=0$, the equality is immediate.
\end{proof}

We will need the following standard closedness criterion for sums of closed convex cones, which is a special case of Rockafellar's closedness theorem; see \cite[Corollary~9.1.2]{rockafellar-convex}.

\begin{lemma} \label{lem: closedness of sum}
Let $K,L \subset \R^n$ be nonempty closed convex cones. If $K \cap (-L) = \{0\}$, then $K+L$ is closed.
\end{lemma}

Let $I \subset \{1,\ldots,n\}$ be a nonempty index set. Define the coordinate projection $\pi_I: \R^n \to \R^{|I|}$ by $\pi_I(x) = (x_i)_{i \in I}$.

For a set-valued map $F: \R^n \rightrightarrows \R^m$, the \emph{Painlev\'e--Kuratowski upper and lower limits} of $F$ are defined by
    \begin{gather*}
    \Limsup_{x \to \Bar x} F(x) := \{y \in \R^m \mid \exists x_k \to \Bar x, \exists y_k \in F(x_k),\ y_k \to y\}, \\
    \Liminf_{x \to \Bar x} F(x) := \{y \in \R^m \mid \forall x_k \to \Bar x,\ \exists y_k \in F(x_k),\ y_k \to y\}, \\
    \Limsup_{\pi_I(x) \xrightarrow{\Omega} \infty} F(x) := \{y \in \R^m \mid \exists x_k \in \Omega \;\; \text{with } \ \|\pi_I(x_k)\| \to \infty, \ \exists y_k \in F(x_k),\ y_k \to y\}, \\
    \Liminf_{\pi_I(x) \xrightarrow{\Omega} \infty} F(x) := \{y \in \R^m \mid \forall x_k \in \Omega \;\; \text{with } \ \|\pi_I(x_k)\| \to \infty, \ \exists y_k \in F(x_k),\ y_k \to y\}.
    \end{gather*}

Throughout the paper our focus is on convex sets and convex functions. The following are the standard tangent and normal cones from convex analysis.

\begin{definition} \label{def: convex cones at finite}
Let $\Omega \subset \R^n$ be a nonempty convex set, and let $\Bar x \in \Omega$.
\begin{enumerate}[\rm (i)]
    \item The \emph{tangent cone} to $\Omega$ at $\Bar x$ is
    \[
    T(\Bar x;\Omega) := \cl\cone(\Omega - \Bar x).
    \]
    \item The \emph{normal cone} to $\Omega$ at $\Bar x$ is
    \[
    N(\Bar x;\Omega) := \big\{v \in \R^n \mid \langle v, x-\Bar x \rangle \le 0, \ \;\ \text{for all}~ x \in \Omega\big\}.
    \]
\end{enumerate}
\end{definition}

For convex sets, $T(\Bar x;\Omega)$ and $N(\Bar x;\Omega)$ are closed convex cones containing the origin. Alongside these, we record the limiting normal cone, in a form that parallels the construction used at infinity in the next section.

\begin{definition} \label{def: limiting normal at finite}
Let $\Omega \subset \R^n$ be a nonempty closed convex set, and let $\Bar x \in \Omega$. The \emph{limiting normal cone} of $\Omega$ at $\Bar x$ is
\[
    N^M(\Bar x;\Omega) := \Limsup_{x \xrightarrow{\Omega} \Bar x} N(x;\Omega).
\]
\end{definition}

These cones admit equivalent representations that motivate the constructions at infinity in Section~\ref{sec: tangent and normal cones at infinity}; see, e.g.,~\cite[Theorems~6.9 and~6.28]{rockafellar-wets} and~\cite[Proposition~1.5]{mordukhovich-vol1}.

\begin{proposition} \label{prop: equivalent normal cones}
Let $\Omega\subset\R^n$ be a nonempty closed convex set, and let $\Bar x\in\Omega$. Then:
\begin{enumerate}[\rm (i)]
    \item $\displaystyle T(\Bar x;\Omega)=\Liminf_{x\xrightarrow{\Omega}\Bar x,\ t\searrow 0}\frac{\Omega-x}{t}$;
    \item $N(\Bar x;\Omega)=T(\Bar x;\Omega)^\circ$;
    \item $N(\Bar x;\Omega)=N^M(\Bar x;\Omega)$.
\end{enumerate}
\end{proposition}

Let $f:\R^n\to\R\cup\{\infty\}$ be an extended-real-valued function. We denote by $\dom f$, $\gph f$, and $\epi f$ the domain, graph, and epigraph of $f$, respectively. We say that $f$ is \emph{proper} if $\dom f\neq\emptyset$. The \emph{Fenchel conjugate} of $f$ is defined by
\[
f^*(u):=\sup_{x\in\R^n}\big\{\langle u,x\rangle-f(x)\big\},
\]
and its domain $\dom f^*$ records the slopes of all affine minorants of $f$. The \emph{recession function} of $f$ is defined by
\[
f^\infty(d):=\lim_{t\to\infty}\frac{f(x_0+td)-f(x_0)}{t}\qquad(x_0\in\dom f\text{ arbitrary}).
\]
Moreover, if $f$ is a proper, lower semicontinuous, convex function, then this limit exists in $\R\cup\{\infty\}$ and is independent of $x_0$, and $f^\infty$ is sublinear and lower semicontinuous with $f^\infty=\sigma_{\dom f^*}$; see \cite[Theorems~8.5 and~13.3]{rockafellar-convex}.

We next recall two classes of functions used in subsequent sections.

\begin{definition} \label{def: polyhedral function}
A convex function $f:\R^n\to\R\cup\{\infty\}$ is called \emph{polyhedral} if its epigraph $\epi f$ is a polyhedral convex set in $\R^{n+1}$. Equivalently, $f$ is the pointwise maximum of finitely many affine functions on its effective domain, which is a polyhedral convex set.
\end{definition}

\begin{definition} \label{def: lipschitz at infinity}
A function $f:\R^n\to\R\cup\{\infty\}$ is \emph{Lipschitz at infinity} with respect to $I$ if there exist constants $L\ge 0$, $\gamma>0$, and $\rho>0$ such that, for every $x\in\dom f$ with $\|\pi_I(x)\|\ge\gamma$, the function $f$ is finite on the ball $B(x;\rho)$ and
\[
|f(y)-f(z)|\le L\|y-z\| \qquad \text{for all } y,z\in B(x;\rho).
\]When $I=\{1,\dots,n\}$, this is the Lipschitzness-at-infinity notion introduced in~\cite{tung23a}.
\end{definition}

\section{Tangent Cones and Normal Cones at Infinity: The Convex Case} \label{sec: tangent and normal cones at infinity}
In this section we study the \emph{Clarke} tangent and normal cones at infinity introduced in~\cite{tung23a,tung23b}, specialized to the convex case. By Proposition~\ref{prop: equivalent normal cones}, the Clarke superscript is redundant at finite points, and we likewise drop it at infinity. The limiting (Mordukhovich) normal cone at infinity, denoted by $N^M$, may strictly differ from $N$ and is treated separately.

In what follows we assume that $\pi_I(\Omega)$ is unbounded. We also define the set
    \[
    B_I := \{x \in \R^n \mid \|\pi_I(x)\| \le 1\}.
    \]

\begin{definition} \label{def: cones at infinity}
Let $\Omega \subset \R^n$ be a nonempty convex set.
\begin{enumerate}[\rm (i)]
\item The \emph{tangent cone to $\Omega$ at infinity} (with respect to $I$) is defined by
    \[
    T(\infty_I;\Omega) := \Liminf_{\pi_I(x) \xrightarrow{\Omega} \infty,\ t \searrow 0} \frac{\Omega - x}{t}.
    \]
\item The \emph{normal cone to $\Omega$ at infinity} (with respect to $I$) is defined by
    \[
    N(\infty_I;\Omega) := \big\{v \in \R^n \mid \langle v,w \rangle \le 0 \ \text{for all } w \in T(\infty_I;\Omega)\big\}.
    \]
\item The \emph{limiting normal cone to $\Omega$ at infinity} (with respect to $I$) is defined by
    \[
    N^M(\infty_I;\Omega) := \Limsup_{\pi_I(x) \xrightarrow{\Omega} \infty} N(x;\Omega).
    \]
\end{enumerate}
\end{definition}
If $I = \{1,\ldots,n\}$, we simply write $T(\infty;\Omega)$, $N(\infty;\Omega)$, and $N^M(\infty;\Omega)$.

We now relate these objects to the classical tangent and normal cones at finite points.

\begin{proposition} \label{prop: cones at infinity}
Let $\Omega \subset \R^n$ be a nonempty closed convex set. Then:
\begin{enumerate}[\rm (i)]
    \item $\displaystyle T(\infty_I;\Omega) = \Liminf_{\pi_I(x) \xrightarrow{\Omega} \infty} T(x;\Omega)$;
    \item $\displaystyle N^M(\infty_I;\Omega) \subset N(\infty_I;\Omega)$.
\end{enumerate}
\end{proposition}

\begin{proof}
(i) Let $w \in T(\infty_I;\Omega) = \Liminf_{\pi_I (x) \xrightarrow{\Omega} \infty, ~ t \searrow 0} \dfrac{\Omega-x}{t}$. Take any sequences $\{x_k\}\subset\Omega$ and $\{t_k\}\subset\R_+$ with $\|\pi_I(x_k)\|\to\infty$ and $t_k\searrow 0$. Then there exist $w_k \to w$ such that $x_k + t_kw_k \in \Omega$ for all $k\in\N$. Hence
    \[
    w_k \in \frac{\Omega-x_k}{t_k} \subset \cone(\Omega-x_k) \subset T(x_k;\Omega),
    \]
where the last inclusion holds by the convexity of $\Omega$. By definition,
$w \in \Liminf_{\pi_I(x) \xrightarrow{\Omega} \infty} T(x;\Omega)$.

To prove the converse inclusion, take
$w\in \Liminf_{\pi_I(x) \xrightarrow{\Omega} \infty} T(x;\Omega)$.
Fix arbitrary sequences $\{x_k\}\subset\Omega$ and $\{t_k\}\subset\R_+$ such that $\|\pi_I(x_k)\|\to\infty$ and $t_k\searrow 0$. 

We will construct $w_k\to w$ with
$x_k+t_kw_k\in\Omega$ for all $k \in \N$. Set $z_k:=x_k+t_kw$, $p_k:=P_\Omega(z_k)$, and
$r_k:=(z_k-p_k)/t_k$. We first prove that $r_k \to 0$. Since
$x_k\in\Omega$, we have $d(z_k;\Omega)\le \|z_k-x_k\|=t_k\|w\|$; hence
$\|r_k\|\le\|w\|$, so $\{r_k\}$ is bounded. It is enough to show that every
cluster point of $\{r_k\}$ is zero. Let $r$ be an arbitrary cluster point and, passing to a subsequence without relabeling, suppose that $r_k\to r$.

Since $p_k=z_k-t_kr_k=x_k+t_k(w-r_k)$, we have $p_k-x_k\to0$. Hence
$\|\pi_I(p_k)\|\to\infty$. By the definition of the inner limit, there are
$u_k\in T(p_k;\Omega)$ such that $u_k\to w$. On the other hand, the projection
characterization gives us $z_k-p_k\in N(p_k;\Omega)$, and therefore
$r_k\in N(p_k;\Omega)$. Thus $\langle r_k,u_k\rangle\le0$, and passing to the
limit gives us $\langle r,w\rangle\le0$.

Now apply the projection characterization once more with $x_k\in\Omega$.
Then $\langle z_k-p_k,x_k-p_k\rangle\le0$. Since $z_k-p_k=t_kr_k$ and
$x_k-p_k=t_k(r_k-w)$, we get $\langle r_k,r_k-w\rangle\le0$. Passing to the
limit yields $\|r\|^2\le \langle r,w\rangle\le0$, so $r=0$. Thus every cluster
point of $\{r_k\}$ is zero, and since $\{r_k\}$ is bounded, $r_k\to0$.

Finally, define $w_k:=(p_k-x_k)/t_k$. Then
$w_k=w-r_k\to w$ and $x_k+t_kw_k=p_k\in\Omega$. Since the sequences $\{x_k\}$ and $\{t_k\}$ were arbitrary, this proves $w\in T(\infty_I;\Omega)$.

(ii) Take any $v \in N^M(\infty_I, \Omega) =\Limsup_{\pi_I(x) \xrightarrow{\Omega} \infty} N(x;\Omega)$. By definition, there exist sequences $\{x_k\}\subset\Omega$ and $\{v_k\}\subset\R^n$ such that $\|\pi_I(x_k)\|\to\infty$, $v_k\to v$, and $v_k\in N(x_k;\Omega)$ for all $k\in\N$. Take any $w \in T(\infty_I, \Omega)$. By (i) we can also choose $w_k \to w$ with $w_k \in T(x_k;\Omega)$ \; for all $k \in \N$. Since $v_k \in N(x_k;\Omega)$, we have $\langle v_k, w_k \rangle \le 0$. Passing to the limit gives us $\langle v,w \rangle \le 0$, hence $v \in N(\infty_I;\Omega)$.
\end{proof}

The example below shows that the equality $N^M(\infty_I;\Omega)=N(\infty_I;\Omega)$ may fail for a convex set $\Omega$.

\begin{example} \label{ex: nonconvex normal cone}
Let \(\Omega=\{(x_1,x_2)\in\R^2\mid x_2\ge |x_1|\}\), and let \(I=\{1,2\}\). By a direct computation, we have
\begin{align*}
    &T(\infty_I;\Omega)
    =\{(w_1,w_2)\in\R^2\mid w_2\ge |w_1|\},\\
    &N(\infty_I;\Omega)
    =\{(v_1,v_2)\in\R^2\mid v_2\le -|v_1|\},\\
    &N^M(\infty_I;\Omega) 
    = \{(t,-t) \in \R^2 \mid t \ge 0\} \cup \{(-t,-t) \in \R^2 \mid t \ge 0\}.
\end{align*}
\end{example}

\begin{remark}\label{rem: limiting normal nonconvex}
For a polyhedral convex cone $\Omega$, $N(\infty_I;\Omega)$ is a closed convex cone. However, $N^M(\infty_I;\Omega)$ may fail to be convex as shown in Example~\ref{ex: nonconvex normal cone}.
\end{remark}

The next example shows that the reverse inclusion in Proposition~\ref{prop: cones at infinity}(i) may fail in the nonconvex case.

\begin{example}\label{ex: nonconvex tangent reverse}
Consider the nonconvex set
    \[
    \Omega:=\bigcup_{n=1}^\infty \Bigl((n,n+\tfrac1n)\times\{0\}\Bigr)\subset \R^2.
    \]
Consider the case where $I = \{1\}$. For every $x\in\Omega$, we have $T(x;\Omega)=\R\times\{0\}$. In particular,
\[
(1,0)\in \Liminf_{\pi_I(x) \xrightarrow{\Omega} \infty} T(x;\Omega).
\]

We will prove that $(1,0) \notin T(\infty_I;\Omega)$. Define
$x_k:=\bigl(k+\frac1k-\frac1{k^2},0\bigr)\in\Omega$ and $t_k:=\frac1k \searrow 0$.
Suppose to the contrary that $(1,0)\in T(\infty_I;\Omega)$. Then there exists $v_k\to(1,0)$ such that
$x_k+t_kv_k\in\Omega$ for all $k\in\N$.
By the definition of $\Omega$ we have
$\frac1k (v_k)_1<\frac1{k^2}$. Equivalently, $(v_k)_1<\frac1k$ and $(v_k)_2=0$,
which contradicts the fact that $v_k\to(1,0)$. Therefore,
    \[    \Liminf_{\pi_I(x)\xrightarrow{\Omega}\infty} T(x;\Omega)\not\subset T(\infty_I;\Omega).
    \]
\end{example}

\subsection{Basic Properties}
We next record a few basic facts about $T(\infty_I;\Omega)$ and $N(\infty_I;\Omega)$ that will be used throughout the paper.

\begin{proposition} \label{prop: basic properties}
Let $\Omega \subset \R^n$ be a nonempty convex set.
\begin{enumerate}[\rm (i)]
    \item $T(\infty_I;\Omega)$ and $N(\infty_I;\Omega)$ are closed convex cones;
    \item
    $N(\infty_I;\Omega) = [T(\infty_I;\Omega)]^{\circ}$ and $T(\infty_I;\Omega) = [N(\infty_I;\Omega)]^{\circ}$;
    \item The multifunction $N(\cdot;\Omega)$ is closed at infinity in the sense that
    \[
    x_k \in \Omega,\ \|\pi_I(x_k)\| \to \infty,\ v_k \in N(x_k;\Omega),\ v_k \to v \ \Longrightarrow\ v \in N(\infty_I;\Omega);
    \]
    \item The multifunction $T(\cdot;\Omega)$ is lower semicontinuous at infinity in the sense that for every open set $V \subset  \R^n$ with $T(\infty_I;\Omega) \cap V \neq \emptyset$, there exists $\gamma > 0$ such that
    \[
    T(x;\Omega) \cap V \neq \emptyset \quad \text{for all} ~ x \in \Omega \setminus (\gamma B_I).
    \]
\end{enumerate}
\end{proposition}

\begin{proof}
(i) We begin with $T(\infty_I;\Omega)$. Let $w_1,w_2 \in T(\infty_I;\Omega)$ and $\lambda_1,\lambda_2 \in \R_+$. Set $\mu := \lambda_1 + \lambda_2$. The case $\mu=0$ is trivial, so assume $\mu>0$. By definition, for any sequences $\{x_k\}\subset\Omega$ and $\{t_k\}\subset\R_+$ with $\|\pi_I(x_k)\|\to\infty$ and $t_k\searrow 0$, there exist $w_{1k} \to w_1$ and $w_{2k} \to w_2$ such that
    \[
    x_k + t_k\mu w_{1k} \in \Omega, \quad x_k + t_k\mu w_{2k} \in \Omega, \quad \;\ \text{for all}~ k\in\N.
    \]
Using the convexity of $\Omega$, we obtain
    \[
    x_k + t_k(\lambda_1w_{1k}+\lambda_2w_{2k}) = \frac{\lambda_1}{\mu}(x_k+t_k\mu w_{1k})+\frac{\lambda_2}{\mu}(x_k+t_k\mu w_{2k}) \in \Omega.
    \]
By definition, $\lambda_1w_1+\lambda_2w_2 \in T(\infty_I;\Omega)$, and hence $T(\infty_I;\Omega)$ is a convex cone. Its closedness is immediate from the definition.

    Because $N(\infty_I;\Omega) = [T(\infty_I;\Omega)]^{\circ}$ is the polar of a set, it is automatically a closed convex cone.

    (ii) The first conclusion is obvious. For the second one, since $T(\infty_I;\Omega)$ is a closed convex cone by (i),
    \[
    [T(\infty_I;\Omega)]^{\circ\circ} = T(\infty_I;\Omega),
    \]
and therefore $[N(\infty_I;\Omega)]^{\circ} = T(\infty_I;\Omega)$.

    (iii) Take any sequence $\{x_k\}\subset\Omega$ with $\|\pi_I(x_k)\|\to\infty$, and take any sequence $\{v_k\}\subset\R^n$ such that $v_k\to v$, where $v_k\in N(x_k;\Omega)$ for every $k\in\N$. Fix any $w \in T(\infty_I;\Omega)$, and choose an arbitrary sequence $\{t_k\} \subset \R_+$ with $t_k \searrow 0$. By the definition of $T(\infty_I;\Omega)$, there exist a sequence $\{w_k\} \subset \R^n$ such that $w_k \to w$ and $x_k + t_kw_k \in \Omega$ for all $k$. Since $v_k\in N(x_k;\Omega)$ and $x_k+t_kw_k\in\Omega$ for all $k \in \N$, we have
    \[
    \langle v_k, (x_k+t_kw_k) - x_k \rangle \le 0,
    \]
    which implies that $\langle v_k, w_k\rangle \le 0$ for all $k \in \N$. Passing to the limit gives us $\langle v, w\rangle \le 0$. Since $w \in T(\infty_I;\Omega)$ is arbitrary, we see that $v \in N(\infty_I;\Omega)$.

    (iv) Fix an open set $V\subset\R^n$ with $T(\infty_I;\Omega)\cap V\neq\emptyset$ and choose $w\in T(\infty_I;\Omega)\cap V$. Since $V$ is open, there exists $\varepsilon>0$ such that $B(w;\varepsilon)\subset V$. Assume for contradiction that for every $\gamma>0$ there exists $x\in\Omega\setminus\gamma B_I$ with $T(x;\Omega)\cap B(w;\varepsilon)=\emptyset$. Then we can find a sequence $\{x_k\}\subset\Omega$ with $\|\pi_I(x_k)\|\to\infty$ such that
    \[
    T(x_k;\Omega)\cap B(w;\varepsilon)=\emptyset \quad \;\ \text{for all}~ k\in\N.
    \]
By the separation theorem, for each $k$ there exists a nonzero $v_k\in\R^n$ satisfying
\begin{equation} \label{eq: separation theorem}
    \sup_{s \in T(x_k;\Omega)} \langle v_k, s \rangle \le \inf_{e \in B(w;\varepsilon)} \langle v_k, e \rangle.
\end{equation}
Normalizing, we may assume $\|v_k\|=1$. Passing to a subsequence if necessary, let $v_k\to v$ for some $v\in\R^n$. 

We claim that \(v_k\in N(x_k;\Omega)\). Indeed, since \(T(x_k;\Omega)\) is a cone, we can easily see that $\langle v_k,s\rangle\le 0$ for all $s\in T(x_k;\Omega)$, which means that $v_k\in [T(x_k;\Omega)]^\circ=N(x_k;\Omega)$. By (iii), we conclude that $v\in N(\infty_I;\Omega)$, and thus $\langle v,w\rangle \le 0$. On the other hand, since $0\in T(x_k;\Omega)$, combining 
\[
\inf_{e\in B(w;\varepsilon)}\langle v_k,e\rangle =\langle v_k,w\rangle-\varepsilon\|v_k\| = \langle v_k, w \rangle - \varepsilon.
\]
with \eqref{eq: separation theorem} gives us $0 \le \langle v_k, w \rangle - \varepsilon$. Letting $k\to\infty$ yields $\langle v,w\rangle \ge \varepsilon$, which is a contradiction.
\end{proof}

\begin{remark}
    A nonconvex version of Proposition~\ref{prop: basic properties} was established in \cite[Corollary~3.7]{tung23a}. Under the convexity of $\Omega$, Proposition~\ref{prop: basic properties} does not require $\Omega$ to be closed or $T(\infty_I;\Omega)$ to be solid as in \cite[Corollary~3.7]{tung23a}.
\end{remark}

We next study the relationship between the recession cone $\Omega^\infty$ and the tangent and normal cones to $\Omega$ at infinity for the convex case.

\begin{definition} (see \cite{nambook}) \label{def: recession}
Let $\Omega\subset\R^n$ be a nonempty set. The \emph{recession cone} of $\Omega$ is defined by
    \[
    \Omega^\infty := \{d\in \R^n \mid x+td \in \Omega, \ \;\ \text{for all}~ x \in \Omega,\ \;\ \text{for all}~ t \ge 0\}.
    \]
When $\Omega$ is closed and convex, we have the presentation
    \[
    \Omega^\infty = \Limsup_{t \searrow 0} t\,\Omega.
    \]
\end{definition}

\begin{proposition} \label{prop: recession inclusion}
Let $\Omega\subset\R^n$ be a nonempty closed convex set. Then
    \[
    \Omega^\infty \subset T(\infty_I;\Omega) \qquad \text{and} \qquad N(\infty_I;\Omega) \subset (\Omega^\infty)^\circ.
    \]
\end{proposition}

\begin{proof}
Fix an arbitrary element $d \in \Omega^\infty$. Take any sequences $\{x_k\}\subset\Omega$ and $\{t_k\}\subset\R_+$ with $\|\pi_I(x_k)\|\to\infty$ and $t_k\searrow 0$. By Definition~\ref{def: recession}, we have $x_k+t_kd \in \Omega$ for all $k \in \N$. It follows from the definition of $T(\infty_I;\Omega)$ that $d \in T(\infty_I;\Omega)$. The second inclusion follows from the first one and Proposition~\ref{prop: basic properties}(ii) by taking polarity.
\end{proof}

\begin{remark}\label{rem: recession strict}
In general, the inclusion $\Omega^\infty \subset T(\infty_I;\Omega)$ could be strict even for a polyhedral convex set as shown in the example below. Let
    \[
    \Omega=\{(x_1,x_2)\in\R^2\mid x_1\ge 0,\ 0\le x_2\le 1\},
    \qquad I=\{1\}.
    \]
Then
    \[
    \Omega^\infty=\R_+\times\{0\}
    \subsetneq
    \R\times\{0\} = T(\infty_I;\Omega).
    \]
\end{remark}

\subsection{The Polyhedral Convex Case} \label{subsec: polyhedral}
In this subsection, we study the tangent and normal cones at infinity to polyhedral convex sets. Recall that a set $\Omega$ in $\R^n$ is said to be  polyhedral convex if it has the presentation
\begin{equation} \label{eq: polyhedral representation}
    \Omega = \bigcap_{j=1}^m \{x\in\R^n \mid \langle a_j, x\rangle\le b_j\}
\end{equation}
where $a_1,\dots,a_m\in\R^n\setminus\{0\}$ and $b_1,\dots,b_m\in\R$. Given $x\in\Omega$, the active index set at $x$ is defined by
\[
    \mathcal{A}(x) := \{j=1,\ldots,m \mid \langle a_j,x\rangle = b_j\}.
\]
For each $j=1,\ldots,m$, we denote
\[
    F_j := \{x\in\Omega\mid \langle a_j,x\rangle = b_j\}.
\]

\begin{definition} \label{def: asymptotic active set}
Define the \emph{asymptotic active index set} of $\Omega$ relative to $I$ by
\[
    J^\infty_I(\Omega) := \{j = 1,\dots,m \mid \pi_I(F_j)\text{ is unbounded}\}.
\]
When $I=\{1,\dots,n\}$, the set $J^\infty_I(\Omega)$ is denoted by $J^\infty(\Omega)$. In this case, it is obvious $\pi_I(F_j)$ is unbounded if and only if $F_j$ is unbounded.
\end{definition}

We continue with a useful fact.

\begin{lemma} \label{lem: asymptotic active}
Consider a nonempty polyhedral convex set $\Omega$ given by \eqref{eq: polyhedral representation}, and let $\{x_k\}\subset\Omega$ satisfy $\|\pi_I(x_k)\|\to\infty$. Then there exists $K\in\N$ such that $\mathcal{A}(x_k)\subset J^\infty_I(\Omega)$ for all $k\ge K$.
\end{lemma}

\begin{proof}
By Definition~\ref{def: asymptotic active set},  $\pi_I(F_j)$ is bounded for each $j \notin J^\infty_I(\Omega)$. Thus, there exists $M_j>0$ such that $\|\pi_I(y)\|\le M_j$ for every $y\in F_j$. Set $M:=\max\{M_j\mid j\notin J^\infty_I(\Omega)\}$ (with $M:=0$ if $J^\infty_I(\Omega)=\{1,\dots,m\}$). Since $\|\pi_I(x_k)\|\to\infty$, there exists $K\in\N$ such that $\|\pi_I(x_k)\|>M$ for all $k\ge K$. Fix $k\ge K$ and take any $j \in\mathcal{A}(x_k)$. Then by definition $x_k\in F_j$. Suppose on the contrary that $j\notin J^\infty_I(\Omega)$. Since $x_k\in F_j$, we see that $\|\pi_I(x_k)\|\le M_j\le M$, which is a contradiction because $\|\pi_I(x_k)\|>M$. Therefore $j\in J^\infty_I(\Omega)$, and hence $\mathcal{A}(x_k)\subset J^\infty_I(\Omega)$.
\end{proof}

Using Lemma~\ref{lem: asymptotic active}, we obtain the following presentations of the tangent and normal cones at infinity for polyhedral convex sets.

\begin{theorem} \label{prop: polyhedral cones at infinity}
Consider a nonempty polyhedral convex set $\Omega$ given by \eqref{eq: polyhedral representation}. Then
\begin{align*}
    T(\infty_I;\Omega) &= \{d\in\R^n\mid \langle a_j,d\rangle\le 0 \text{ for all } j\in J^\infty_I(\Omega)\},\\
    N(\infty_I;\Omega) &= \sum_{j\in J^\infty_I(\Omega)} \R_+ a_j.
\end{align*}
\end{theorem}

\begin{proof}
Define the set 
\begin{equation} \label{eq: formular for polyhedral}
	D:=\{d\in\R^n\mid \langle a_j,d\rangle\le 0 \text{ for all } j\in J^\infty_I(\Omega)\}.
\end{equation}

\noindent\emph{Inclusion $T(\infty_I;\Omega)\subset D$.} Let $d\in T(\infty_I;\Omega)$ and fix $j_0\in J^\infty_I(\Omega)$. By Definition~\ref{def: asymptotic active set}, $\pi_I(F_{j_0})$ is unbounded, so there exists a sequence $\{x_k\}\subset F_{j_0}$ with $\|\pi_I(x_k)\|\to\infty$. By Proposition~\ref{prop: cones at infinity}(i), there exists $d_k\to d$ with $d_k\in T(x_k;\Omega)$. Since $x_k\in F_{j_0}$ means that $\langle a_{j_0},x_k\rangle=b_{j_0}$, we have $j_0\in\mathcal{A}(x_k)$. The standard polyhedral tangent-cone formula \cite[Theorem~6.46]{rockafellar-wets} gives us $\langle a_{j_0},d_k\rangle\le 0$ for all $k$. Passing to the limit yields $\langle a_{j_0},d\rangle\le 0$. Since $j_0\in J^\infty_I(\Omega)$ is arbitrary, we see that $d\in D$.

\smallskip
\noindent\emph{Inclusion $D\subset T(\infty_I;\Omega)$.} Let $d\in D$ and fix any sequence $\{x_k\}\subset\Omega$ with $\|\pi_I(x_k)\|\to\infty$. By Lemma~\ref{lem: asymptotic active}, there exists $K\in\N$ such that $\mathcal{A}(x_k)\subset J^\infty_I(\Omega)$ for all $k\ge K$. For a polyhedral convex set $\Omega$,
\[
    T(x_k;\Omega) = \{v\in\R^n\mid \langle a_j,v\rangle\le 0 \text{ for all } j\in\mathcal{A}(x_k)\}.
\]
Since $d\in D$ and $\mathcal{A}(x_k)\subset J^\infty_I(\Omega)$, we have $d\in T(x_k;\Omega)$ for all $k\ge K$. Taking the constant sequence $d_k:=d$ in the inner-limit characterization from Proposition~\ref{prop: cones at infinity}(i), we conclude $d\in T(\infty_I;\Omega)$.

\smallskip
\noindent\emph{Normal cone formula.} By Proposition~\ref{prop: basic properties}(ii), we have $N(\infty_I;\Omega) = T(\infty_I;\Omega)^\circ = D^\circ$, where $D$ is defined by \eqref{eq: formular for polyhedral}. By Farkas' Lemma \cite[Corollary~22.3.1]{rockafellar-convex}, the polar of $D$ is given by
\[
    D^\circ = \sum_{j\in J^\infty_I(\Omega)} \R_+ a_j,
\]
which completes the proof.
\end{proof}

\begin{remark} \label{rem: polyhedral recession equality}
It is well-known that the recession cone of $\Omega$ admits the following presentation
\[
    \Omega^\infty = \{d\in\R^n\mid \langle a_j,d\rangle\le 0 \text{ for all } j = 1,\dots,m\}.
\]
Therefore, by Proposition~\ref{prop: recession inclusion} we have the inclusion $\Omega^\infty\subset T(\infty_I;\Omega)$, which also follows from Theorem~\ref{prop: polyhedral cones at infinity}.

In the polyhedral convex setting, it is obvious that the equality $T(\infty_I;\Omega)=\Omega^\infty$ holds if and only if we have the implication
\begin{equation} \label{eq: condition for recession equality}
	\big[\langle a_l,d\rangle\le 0 \quad \text{for all} \quad l \in J^\infty_I(\Omega)\big] \qquad \Longrightarrow \qquad \big[\langle a_j,d\rangle\le 0 \quad \text{for all} \quad j \notin J^\infty_I(\Omega)\big].
\end{equation}
In particular, this equality holds if $J^\infty_I(\Omega)=\{1,\dots,m\}$. For instance, let
\[
    \Omega=\{(x_1,x_2)\in\R^2\mid x_1\ge 0,\ x_2\ge 0,\ x_1+x_2\ge 1\}
\]
with $I=\{1,2\}$. We can verify that $J^\infty_I(\Omega)=\{1,2\}\subsetneq\{1,2,3\}$. Observe also that if $-d_1\le 0$ and $-d_2\le 0$, then $-d_1-d_2\le 0$, so \eqref{eq: condition for recession equality} is satisfied. Therefore, $T(\infty_I;\Omega)=\Omega^\infty=\R_+^2$.
\end{remark}


\subsection{Calculus Rules}
In this subsection, we establish calculus rules for the tangent and normal cones at infinity under products, intersections, and linear mappings.

\begin{proposition} \label{prop: normal product rules}
Let $\Omega_1 \subset \R^n$ and $\Omega_2 \subset \R^m$ be nonempty convex sets. For two nonempty index sets $I \subset \{1,\ldots,n\}$ and $J \subset \{1,\ldots,m\}$, put $K := I \cup \{n + j \mid j \in J\} \subset \{1,\ldots,n+m\}$, and denote the coordinate projections by
\[
\pi_I: \R^n \to \R^{|I|}, \quad \pi_J: \R^m \to \R^{|J|}, \quad \pi_K: \R^n \times \R^m \to \R^{|K|}.
\]
Assume that $\pi_I(\Omega_1)$ and $\pi_J(\Omega_2)$ are unbounded. Then:
\begin{enumerate}[\rm (i)]
    \item $T(\infty_K;\Omega_1 \times \Omega_2) \subset T(\infty_I;\Omega_1) \times T(\infty_J;\Omega_2)$;
    \item $N(\infty_I;\Omega_1) \times N(\infty_J;\Omega_2) \subset N(\infty_K;\Omega_1 \times \Omega_2)$.
\end{enumerate}
\end{proposition}

\begin{proof}
(i) Let $(w_1,w_2) \in T(\infty_K;\Omega_1 \times \Omega_2)$. Fix any sequences $\{x_{1k}\}\subset\Omega_1$, $\{x_{2k}\}\subset\Omega_2$, and $\{t_k\}\subset\R_+$ such that $\|\pi_I(x_{1k})\|\to\infty$, $\|\pi_J(x_{2k})\|\to\infty$, and $t_k\searrow 0$. Then $\|\pi_K(x_{1k},x_{2k})\| \to \infty$. By definition, there exist sequences $\{w_{1k}\}\subset\R^n$ and $\{w_{2k}\}\subset\R^m$ such that $w_{1k}\to w_1$, $w_{2k}\to w_2$, and
    \[
    (x_{1k},x_{2k}) + t_k(w_{1k},w_{2k}) \in \Omega_1 \times \Omega_2, \quad \;\ \text{for all}~ k\in\N.
    \]
Equivalently, we have $x_{1k} + t_kw_{1k} \in \Omega_1$ and $x_{2k} + t_kw_{2k} \in \Omega_2$ for all $k$, which implies that $w_1 \in T(\infty_I;\Omega_1)$ and $w_2 \in T(\infty_J;\Omega_2)$.

(ii) Let $v_1 \in N(\infty_I;\Omega_1)$ and $v_2 \in N(\infty_J;\Omega_2)$. For any $w_1 \in T(\infty_I;\Omega_1)$ and $w_2 \in T(\infty_J;\Omega_2)$ we have
    \[
    \langle (v_1,v_2), (w_1,w_2) \rangle = \langle v_1, w_1 \rangle + \langle v_2, w_2 \rangle \le 0.
    \]
By (i), this inequality holds for all $(w_1,w_2) \in T(\infty_K;\Omega_1 \times \Omega_2)$. Hence $(v_1,v_2) \in N(\infty_K;\Omega_1 \times \Omega_2)$.
\end{proof}

\begin{remark} \label{rem: polyhedral product}
 Suppose that, for $\ell=1,2$, the set $\Omega_\ell=\bigcap_{s=1}^{m_\ell}\{x_\ell\mid \langle a_s^\ell,x_\ell\rangle\le b_s^\ell\}$ is a polyhedral convex set with $\pi_I(\Omega_1)$ and $\pi_J(\Omega_2)$ unbounded. Then $\Omega_1\times\Omega_2\subset\R^{n+m}$ is a polyhedral convex set, defined by $m_1+m_2$ constraints: the type-1 constraints $\langle(a_s^1,0),(x_1,x_2)\rangle\le b_s^1$ ($s=1,\dots,m_1$) and the type-2 constraints $\langle(0,a_s^2),(x_1,x_2)\rangle\le b_s^2$ ($s=1,\dots,m_2$). In this case, the tangent and normal cones of $\Omega_1\times\Omega_2$ at infinity can be presented in terms of the recession cones of $\Omega_1$ and $\Omega_2$ as shown below.

Since $\pi_I(\Omega_1)$ is unbounded, we have $\Omega_1$ is unbounded; similarly $\Omega_2$. The face of $\Omega_1\times\Omega_2$ associated with a type-1 constraint $s$ is $F_s^1\times\Omega_2$, whose $\pi_K$-projection is $\pi_I(F_s^1)\times\pi_J(\Omega_2)$. This is unbounded because $\pi_J(\Omega_2)$ is unbounded; hence the type-1 constraint $s$ lies in $J^\infty_K(\Omega_1\times\Omega_2)$. By the symmetric argument, every type-2 constraint also lies in $J^\infty_K(\Omega_1\times\Omega_2)$. Therefore $J^\infty_K(\Omega_1\times\Omega_2)$ exhausts all $m_1+m_2$ defining constraints, and Theorem~\ref{prop: polyhedral cones at infinity} gives us
\[
    T(\infty_K;\Omega_1\times\Omega_2) = \left\{(d_1,d_2) \in \R^2 \;\middle|
    \begin{array}{l}
        \langle a_{s_1}^1,d_1\rangle\le 0, \quad s_1=1,\dots,m_1,\\
        \langle a_{s_2}^2,d_2\rangle\le 0, \quad s_2=1,\dots,m_2
    \end{array}\right\} = \Omega_1^\infty\times\Omega_2^\infty,
\]
\[
    N(\infty_K;\Omega_1\times\Omega_2) = \sum_s \R_+(a_s^1,0) + \sum_s \R_+(0,a_s^2) = (\Omega_1^\infty)^\circ\times(\Omega_2^\infty)^\circ.
\]
\end{remark}

\begin{remark} \label{rem: product strict}
The inclusions in Proposition~\ref{prop: normal product rules}(i) and (ii) may be strict even in the polyhedral convex case. For instance, let $n=m=1$, $\Omega_1=\Omega_2=\R_+$, and $I=J=\{1\}$, so that $K=\{1,2\}$. By a direct computation,
\begin{align*}
    &T(\infty_K;\Omega_1 \times \Omega_2) = \R^2_+, \quad \text{while} \quad T(\infty_I;\Omega_1) \times T(\infty_J;\Omega_2) = \R^2, \\
    &N(\infty_K;\Omega_1 \times \Omega_2) = -\R^2_+, \quad \text{while} \quad N(\infty_I;\Omega_1) \times N(\infty_J;\Omega_2) = \{(0,0)\}.
\end{align*}
\end{remark}

We next introduce an asymptotic qualification condition, formulated in terms of normal cones at infinity, as follows.

\begin{lemma}\label{lem: normal separation at infinity}
Let $\Omega_1,\Omega_2\subset\R^n$ be closed convex sets such that $\Omega_1\cap\Omega_2$ is nonempty and unbounded. Assume that
    \begin{equation} \label{eq: normal separation at infinity}
    N(\infty_I;\Omega_1)\cap\bigl(-N(\infty_I;\Omega_2)\bigr)=\{0\}.
    \end{equation}
Then there exist constants $\alpha>0$ and $\gamma>0$ such that, for every $x\in \Omega_1\cap\Omega_2\setminus\gamma B_I$, the estimate
\begin{equation} \label{eq: tangent intersection error bound}
    {\rm d}\bigl(w,\,T(x;\Omega_1)\cap T(x;\Omega_2)\bigr)
    \le \alpha\Bigl({\rm d}\bigl(w, T(x;\Omega_1)\bigr)+{\rm d}\bigl(w, T(x;\Omega_2)\bigr)\Bigr)
\end{equation}
holds for all $w\in\R^n$, and
\begin{equation} \label{eq: normal separation at finite}
    N(x;\Omega_1)\cap\bigl(-N(x;\Omega_2)\bigr)=\{0\}.
\end{equation}
\end{lemma}

\begin{proof}
We claim that there exist $\eta>0$ and $\gamma>0$ such that, for every $x\in\Omega_1\cap\Omega_2\setminus\gamma B_I$,
\begin{equation} \label{eq: normal non cancellation}
    \|u+v\|\ge \eta\bigl(\|u\|+\|v\|\bigr) \quad \text{for all } u\in N(x;\Omega_1),\ v\in N(x;\Omega_2).
\end{equation}
Suppose to the contrary that \eqref{eq: normal non cancellation} fails. Then we can find sequences $\{x_k\}\subset\Omega_1\cap\Omega_2$, $\{u_k\}\subset\R^n$, and $\{v_k\}\subset\R^n$ with $\|\pi_I(x_k)\|>k$, $u_k\in N(x_k;\Omega_1)$, and $v_k\in N(x_k;\Omega_2)$ for all $k\in\N$, such that
    \[
    \|u_k+v_k\|<\frac1k\bigl(\|u_k\|+\|v_k\|\bigr).
    \]
Define the vectors
    \[
    \widehat u_k:=\frac{u_k}{\|u_k\|+\|v_k\|},\qquad \widehat v_k:=\frac{v_k}{\|u_k\|+\|v_k\|}.
    \]
Then $\widehat u_k\in N(x_k;\Omega_1)$, $\widehat v_k\in N(x_k;\Omega_2)$ and
    \[
    \|\widehat u_k\|+\|\widehat v_k\|=1,\qquad \|\widehat u_k+\widehat v_k\|<\frac1k.
    \]
Passing to a subsequence without relabeling, we may assume that $\widehat u_k\to\bar u$ and $\widehat v_k\to\bar v$. Consequently,
    \[
    \bar u+\bar v=0,\qquad \|\bar u\|+\|\bar v\|=1,
    \]
which implies that $\bar u\neq 0$. Since $\|\pi_I(x_k)\|\to\infty$, Proposition~\ref{prop: basic properties}(iii) yields that $\bar u\in N(\infty_I;\Omega_1)$ and $\bar v\in N(\infty_I;\Omega_2)$. Using $\bar v=-\bar u$, we obtain
    \[
    0\neq\bar u\in N(\infty_I;\Omega_1)\cap\bigl(-N(\infty_I;\Omega_2)\bigr),
    \]
which contradicts \eqref{eq: normal separation at infinity}.

We now derive \eqref{eq: tangent intersection error bound} and \eqref{eq: normal separation at finite} from \eqref{eq: normal non cancellation}. Fix $x\in\Omega_1\cap\Omega_2\setminus\gamma B_I$.

For \eqref{eq: tangent intersection error bound}, define the sets
    \[
    K:=T(x;\Omega_1),\qquad L:=T(x;\Omega_2),\qquad C:=K\cap L.
    \]
By Proposition~\ref{prop: equivalent normal cones}(ii) and Lemma~\ref{lem: polarity rules}(iv), we have $K^\circ=N(x;\Omega_1)$, $L^\circ=N(x;\Omega_2)$, and $C^\circ=\cl(K^\circ+L^\circ)$. Let $s\in C^\circ$ with $\|s\|\le 1$. Choose sequences $\{u_k\}\subset K^\circ$ and $\{v_k\}\subset L^\circ$ such that $u_k+v_k\to s$. By \eqref{eq: normal non cancellation}, it follows that
    \[
    \|u_k\|+\|v_k\|\le \frac{1}{\eta}\,\|u_k+v_k\|.
    \]
Hence $\|u_k\|+\|v_k\|\le 2/\eta$ for all sufficiently large $k$. Next, by Lemma~\ref{lemma: moreau}, for every $w\in\R^n$ we have
    \[
    \langle u_k,w\rangle\le \|u_k\|{\rm d}(w; K),\qquad \langle v_k,w\rangle\le \|v_k\|{\rm d}(w; L).
    \]
Therefore, for all large $k$, we deduce that
    \[
    \langle u_k+v_k,w\rangle\le \frac{2}{\eta}\bigl({\rm d}(w; K)+{\rm d}(w; L)\bigr).
    \]
Letting $k\to\infty$ yields the same bound for $\langle s,w\rangle$. Taking the supremum over $s\in C^\circ$ with $\|s\|\le 1$ and using Lemma~\ref{lemma: moreau}, we conclude that
\[
{\rm d}(w; C)\le \frac{2}{\eta}\bigl({\rm d}(w; K)+{\rm d}(w; L)\bigr).
\]
This is exactly \eqref{eq: tangent intersection error bound} with $\alpha:=2/\eta$.

For \eqref{eq: normal separation at finite}, let $u\in N(x;\Omega_1) \cap \bigl(-N(x;\Omega_2)\bigr)$. Writing $v:=-u$, we have $v\in N(x;\Omega_2)$. Applying \eqref{eq: normal non cancellation} to the pair $(u,v)$ gives us
    \[
    0=\|u+v\|\ge \eta\bigl(\|u\|+\|v\|\bigr)=2\eta\|u\|.
    \]
Thus, $u = 0$, which ends the proof.
\end{proof}

Using Lemma~\ref{lem: normal separation at infinity}, we now derive the following intersection rule for tangent and normal cones at infinity.

\begin{proposition} \label{prop: normal intersection rule}
Let $\Omega_1, \Omega_2 \subset \R^n$ be closed convex sets such that $\Omega_1 \cap \Omega_2$ is nonempty and unbounded. Assume that \eqref{eq: normal separation at infinity} holds. Then:
\begin{enumerate}[\rm (i)]
    \item $T(\infty_I;\Omega_1) \cap T(\infty_I;\Omega_2) \subset T(\infty_I;\Omega_1 \cap \Omega_2)$;
    \item $N(\infty_I;\Omega_1 \cap \Omega_2) \subset N(\infty_I;\Omega_1) + N(\infty_I;\Omega_2)$.
\end{enumerate}
\end{proposition}

\begin{proof}
(i) Fix $w \in T(\infty_I;\Omega_1)\cap T(\infty_I;\Omega_2)$ and let $\{x_k\}\subset\Omega_1\cap\Omega_2$ be any sequence with $\|\pi_I(x_k)\|\to\infty$. With $\gamma>0$ as in Lemma~\ref{lem: normal separation at infinity}, we have $x_k\in\Omega_1\cap\Omega_2\setminus\gamma B_I$ for all $k$ sufficiently large. Since $w \in T(\infty_I;\Omega_\ell)$ for $\ell=1,2$, we have ${\rm d}\bigl(w, T(x_k;\Omega_1)\bigr)\to 0$ and ${\rm d}\bigl(w, T(x_k;\Omega_2)\bigr)\to 0$. Applying \eqref{eq: tangent intersection error bound} with $x=x_k$ yields
    \[
    {\rm d}\bigl(w,\, T(x_k;\Omega_1)\cap T(x_k;\Omega_2)\bigr)\to 0.
    \]
By \eqref{eq: normal separation at finite}, together with \cite[Corollary~2.59]{nambook}, we have
    \[
    T(x_k;\Omega_1)\cap T(x_k;\Omega_2)=T(x_k;\Omega_1\cap\Omega_2).
    \]
Therefore, ${\rm d}\bigl(w,\,T(x_k;\Omega_1\cap\Omega_2)\bigr)\to 0$. Equivalently, we can choose $w_k\in T(x_k;\Omega_1\cap\Omega_2)$ with $w_k\to w$. By Proposition~\ref{prop: cones at infinity}(i), we conclude that $w\in T(\infty_I;\Omega_1\cap\Omega_2)$.

(ii) follows by taking polars in (i) and using Lemma~\ref{lem: closedness of sum} to drop the closure.
\end{proof}

\begin{remark} \label{rem: polyhedral intersection qualification}
If $\Omega_\ell=\bigcap_{j=1}^{m_\ell}\{x\in\R^n\mid \langle a_j^\ell,x\rangle\le b_j^\ell\}$ is a polyhedral convex set for $\ell=1,2$, then $N(\infty_I;\Omega_\ell)=\sum_{j\in J^\infty_I(\Omega_\ell)}\R_+ a_j^\ell$ by Theorem~\ref{prop: polyhedral cones at infinity}. In this case, we can verify that the following \emph{positive linear independence} is a sufficient condition for \eqref{eq: normal separation at infinity}:
\[
    \sum_{j\in J^\infty_I(\Omega_1)}\alpha_j a_j^1 + \sum_{j\in J^\infty_I(\Omega_2)}\beta_j a_j^2 = 0,\ \alpha_j,\beta_j\ge 0
    \;\Longrightarrow\; \alpha_j=\beta_j=0.
\]
\end{remark}

\begin{remark} \label{rem: normal intersection rule}
Consider the following \emph{error bound at infinity}: there exist constants $\alpha>0$, $\gamma>0$, and $\varepsilon>0$ such that
\begin{equation} \label{eq: distance of intersection}
{\rm d}(x; \Omega_1 \cap \Omega_2) \le \alpha\,{\rm d}(x; \Omega_2),
\quad \;\ \text{for all}~ x \in \Omega_1 \setminus \gamma B_I ~ \text{with} ~ {\rm d}(x; \Omega_2) < \varepsilon.
\end{equation}
We will show that, under \eqref{eq: distance of intersection}, Proposition~\ref{prop: normal intersection rule} can be proved without invoking \eqref{eq: normal separation at infinity}. It suffices to check part~(i). Fix $w \in T(\infty_I;\Omega_1) \cap T(\infty_I;\Omega_2)$. Let $\{t_k\}\subset\R_+$ and $\{x_k\}\subset\Omega_1\cap\Omega_2$ be any sequences with $t_k\searrow 0$ and $\|\pi_I(x_k)\|\to\infty$. By definition, there exist sequences $\{w_{1k}\}\subset\R^n$ and $\{w_{2k}\}\subset\R^n$ such that $w_{1k} \to w$, $w_{2k} \to w$, and, for every $k \in \N$,
    \[
    x_k + t_kw_{1k} \in \Omega_1, \quad \text{and} \quad x_k + t_kw_{2k} \in \Omega_2.
    \]
Applying \eqref{eq: distance of intersection} to $x:=x_k+t_kw_{1k}$, for $k$ sufficiently large we get
    \[
    {\rm d}(x_k + t_kw_{1k}, \Omega_1 \cap \Omega_2) \le \alpha \, {\rm d}(x_k + t_kw_{1k}, \Omega_2).
    \]
Hence there exists $\overline{x}_k \in \Omega_1 \cap \Omega_2$ such that
    \[
    \|\overline{x}_k - x_k - t_kw_{1k}\| \le \alpha\|(x_k + t_kw_{2k}) - (x_k + t_kw_{1k})\| = \alpha t_k\|w_{2k} - w_{1k}\|.
    \]
Dividing the last inequality by $t_k$ and putting $w_k := (\overline{x}_k - x_k)/t_k$ yields
    \[
    \|w_k - w_{1k}\| \le \alpha\|w_{2k} - w_{1k}\|.
    \]
Letting $k \to \infty$ gives us $w_k \to w$. Moreover, $\overline{x}_k = x_k + t_kw_k \in \Omega_1 \cap \Omega_2$, and therefore $w \in T(\infty_I;\Omega_1 \cap \Omega_2)$.
\end{remark}

\begin{remark} \label{rem: tung-son comparison}
Our proof of Proposition~\ref{prop: normal intersection rule} differs from that of \cite[Proposition~3.8]{tung23b}: rather than working directly with normal cones at infinity, we first prove the tangent-cone inclusion and then pass to the normal cones by polarity, obtaining the finite-level condition \eqref{eq: normal separation at finite} from \eqref{eq: normal separation at infinity}. This gives us a shorter and more transparent proof.
\end{remark}

We next describe how tangent and normal cones at infinity behave under a linear map $A$.

\begin{proposition} \label{prop: normal linear image}
Let $\Omega\subset\R^n$ be a nonempty unbounded convex set, and let $A:\R^n\to\R^m$ be linear.
Fix nonempty index sets $I\subset\{1,\dots,n\}$ and $J\subset\{1,\dots,m\}$, and denote the projections by $\pi_I:\R^n\to\R^{|I|}$ and $\pi_J:\R^m\to\R^{|J|}$.
Assume that $\pi_J(A(\Omega))$ is unbounded and that
\begin{equation}\label{eq: compatibility infinity}
x_k\in\Omega,\ \|\pi_J(Ax_k)\|\to\infty
\quad\Longrightarrow\quad
\|\pi_I(x_k)\|\to\infty.
\end{equation}
Then:
\begin{enumerate}[\rm (i)]
    \item if $v\in T(\infty_I;\Omega)$, then $Av\in T(\infty_J;A(\Omega))$;

    \item if $w\in N(\infty_J;A(\Omega))$, then $A^\top w\in N(\infty_I;\Omega)$;

    \item if, in addition, $m=n$, $A$ is invertible, and
    \begin{equation}\label{eq: converse compatibility infinity}
        x_k\in\Omega,\ \|\pi_I(x_k)\|\to\infty
        \quad\Longrightarrow\quad
        \|\pi_J(Ax_k)\|\to\infty,
    \end{equation}
    then the converses of {\rm(i)} and {\rm(ii)} also hold. Equivalently,
    \[
    A\big(T(\infty_I;\Omega)\big)=T(\infty_J;A(\Omega))
    \quad\text{and}\quad
    A^\top N(\infty_J;A(\Omega))=N(\infty_I;\Omega).
    \]

\end{enumerate}
\end{proposition}

\begin{proof}
(i) Let $v\in T(\infty_I;\Omega)$. We prove that $Av\in T(\infty_J;A(\Omega))$. Fix arbitrary sequences $\{y_k\}\subset A(\Omega)$ and $\{t_k\}\subset\R_+$ with $\|\pi_J(y_k)\|\to\infty$ and $t_k\searrow 0$. We choose a sequence $\{x_k\}\subset\Omega$ such that $y_k=Ax_k$ for all $k$. By \eqref{eq: compatibility infinity}, $\|\pi_I(x_k)\|\to\infty$. Since $v\in T(\infty_I;\Omega)$, there exists a sequence $\{v_k\}\subset\R^n$ such that $v_k\to v$ and $x_k+t_kv_k\in\Omega$ for all $k$. Applying $A$ gives us
\[
y_k+t_kAv_k = A(x_k+t_kv_k)\in A(\Omega) \quad\text{for all } k,
\]
with $Av_k\to Av$. Hence $Av\in T(\infty_J;A(\Omega))$.

\smallskip
\noindent
(ii) Let $w\in N(\infty_J;A(\Omega))$. By definition, $\langle w,s\rangle\le 0$ for all $s\in T(\infty_J;A(\Omega))$. By (i), we have $Av\in T(\infty_J;A(\Omega))$ for every $v\in T(\infty_I;\Omega)$, so
\[
\langle A^\top w,v\rangle=\langle w,Av\rangle\le 0 \quad\text{for all } v\in T(\infty_I;\Omega).
\]
Hence $A^\top w\in N(\infty_I;\Omega)$.

\smallskip
\noindent
(iii) Assume in addition that $m=n$, $A$ is invertible, and \eqref{eq: converse compatibility infinity} holds. Applying (i) and (ii) to the set $A(\Omega)$ and the linear map $A^{-1}$, we have
    \[
    A^{-1}\big(T(\infty_J;A(\Omega))\big)\subset T(\infty_I;\Omega) \quad\text{and}\quad (A^{-1})^\top N(\infty_I;\Omega) \subset N(\infty_J;A(\Omega)).
    \]
This implies that $T(\infty_J;A(\Omega))\subset A\big(T(\infty_I;\Omega)\big)$ and $N(\infty_I;\Omega)\subset A^\top N(\infty_J;A(\Omega))$. Combining with (i) and (ii), we obtain the desired equalities.
\end{proof}

\begin{remark} \label{rem: full-coord linear image}
We distinguish two situations for \eqref{eq: compatibility infinity} and \eqref{eq: converse compatibility infinity}.

\smallskip
\noindent\emph{Full coordinates with $A$ invertible.} Suppose that $m=n$, $I=J=\{1,\dots,n\}$, and $A$ is invertible. Since $A$ is a linear bijection of $\R^n$, both $A$ and $A^{-1}$ are continuous, so there exist constants $c,C>0$ with
\[
c\|x\|\le\|Ax\|\le C\|x\|\qquad \;\ \text{for all}~ x\in\R^n.
\]
Consequently $\|Ax_k\|\to\infty$ if and only if $\|x_k\|\to\infty$, and both \eqref{eq: compatibility infinity} and \eqref{eq: converse compatibility infinity} hold automatically. This is the setting in which Proposition~\ref{prop: chain rule linear} below applies.

\smallskip
\noindent\emph{Partial coordinates.} For partial $I$ or $J$, the conditions can fail even when $A$ is the identity. Take $n=m=2$, $A=\mathrm{Id}$, $I=\{1\}$, $J=\{2\}$, and $x_k=(0,k)$. Then
\[
\|\pi_J(Ax_k)\|=\|\pi_J(x_k)\|=k\to\infty,\qquad \|\pi_I(x_k)\|=0,
\]
which violates \eqref{eq: compatibility infinity}. Thus the conditions are nontrivial.
\end{remark}

The next example shows that the tangent image rule in Proposition~\ref{prop: normal linear image} may fail in the nonconvex case, so convexity of $\Omega$ is essential.

\begin{example}\label{ex: linear-image-nonconvex}
Consider the following nonconvex sets
    \[
    \Omega:=\Omega_1\cup\Omega_2\subset\R^2,
    \qquad
    \Omega_1:=\bigcup_{n=1}^\infty \bigl((-\infty,-n]\times\{n\}\bigr),
    \qquad
    \Omega_2:=\bigcup_{n=1}^\infty \left(\left[n,n+\dfrac1n\right]\times\{0\}\right).
    \]
Fix $I=\{2\}$, $J=\{1\}$, and define a linear map $A:\R^2\to\R$ by $A(x_1,x_2)=x_1$.

Set $v:=(-1,0)$. We claim that $v\in T(\infty_I;\Omega)$.
Indeed, let $x_k=(a_k,n_k)\in\Omega$ satisfy $n_k\to\infty$, and let $t_k\searrow0$. It is clear that $x_k\in\Omega_1$, hence $a_k\le -n_k$ for all $k \in \N$. Taking $v_k:=v$, we have $x_k+t_kv_k=(a_k-t_k,n_k)\in\Omega_1\subset\Omega$ for all $k \in \N$. Thus $v\in T(\infty_I;\Omega)$.

On the other hand,
    \[
    A(\Omega) = (-\infty,-1] \cup \bigcup_{n=1}^\infty \left[n,n+\dfrac1n\right].
    \]
We show that $Av=-1\notin T(\infty_J;A(\Omega))$.
To see this, set $y_k:=k+\frac1{k^3}\in A(\Omega)$ and $t_k:=\frac 1k \searrow 0$. Suppose to the contrary that $-1\in T(\infty_J;A(\Omega))$. Then there exists a sequence $\{u_k\}\subset\R$ such that $u_k\to -1$ and $y_k+t_ku_k\in A(\Omega)$ for all $k$. For all large $k$, we have $u_k\in(-2,0)$, hence
    \[
    y_k+t_ku_k = k+\frac1{k^3}+\frac{u_k}{k}\in \left(k-\dfrac12,k+\dfrac12\right).
    \]
Moreover, for all large $k$,
    \[
    A(\Omega)\cap \left(k-\dfrac12,k+\dfrac12\right)=\left[k,k+\dfrac1k\right],
    \]
so $y_k+t_ku_k\ge k$. This means that $u_k\ge -\frac1{k^2}$ for all large $k$, which contradicts the fact that $u_k\to -1$.
Thus $Av\notin T(\infty_J;A(\Omega))$, and consequently
    \(
    A\bigl(T(\infty_I;\Omega)\bigr)\not\subset T(\infty_J;A(\Omega)).
    \)
\end{example}

\section{Subdifferential at Infinity} \label{sec: subdifferential at infinity}
In this section we study the subdifferential and singular subdifferential at infinity for proper, lower semicontinuous, convex functions, introduced in \cite{tung23a, tung23b}. Working through the normal cone at infinity to the epigraph, we derive their basic properties and calculus rules and relate them to classical convex-analytic objects.

Throughout, let $f:\R^n \to \R \cup \{\infty\}$ be a proper, lower semicontinuous, convex function with $\dom f$ unbounded. Fix $I=\{1,\ldots,n\}$, and let $\pi: \R^n\times\R \to \R^n$ be the projection $\pi(x,y)=x$. For $x\in\dom f$, recall that the (convex) subdifferential and the singular subdifferential of $f$ are defined by
\begin{align*}
    \partial f(x) := \left\{u \in\R^n \ \mid\ (u,-1) \in N\big((x,f(x));\epi f\big)\right\}, \\
    \partial^{\infty} f(x) := \left\{u \in\R^n \ \mid\ (u,0) \in N\big((x,f(x));\epi f\big)\right\}.
\end{align*}

We recall their counterparts at infinity as follows.

\begin{definition}\label{def: subdiff infinity}
Let $f:\R^n \to \R \cup \{\infty\}$ be a proper, lower semicontinuous, convex function with $\dom f$ unbounded.
\begin{enumerate}[\rm (i)]
    \item The \emph{subdifferential of $f$ at infinity} is
    \[
    \partial f(\infty) := \left\{u \in\R^n \ \middle|\ (u,-1) \in N(\infty_I;\epi f)\right\}.
    \]

    \item The \emph{singular subdifferential of $f$ at infinity} is
    \[
    \partial^{\infty} f(\infty) := \left\{u \in\R^n \ \middle|\ (u,0) \in N(\infty_I;\epi f)\right\}.
    \]

    \item The \emph{limiting subdifferential of $f$ at infinity} is (see \cite{tung23b})
    \[
    \partial^{M} f(\infty) := \left\{u \in\R^n \ \middle|\ (u,-1) \in N^M(\infty_I;\epi f)\right\}.
    \]
\end{enumerate}
\end{definition}

\begin{example} \label{ex: indicator}
Let \(\Omega \subset \R^n\) be a nonempty closed convex set, and consider its indicator function \(\delta_\Omega\). Then
    \[
    \partial \delta_\Omega(\infty) = \partial^{\infty} \delta_\Omega(\infty) = N(\infty;\Omega).
    \]
Indeed, since $\epi\delta_\Omega=\Omega\times \R_+$, we have
    \[
    T(\infty_I;\epi\delta_\Omega) = T(\infty;\Omega) \times \R_+,
    \qquad
    N(\infty_I;\epi\delta_\Omega)=N(\infty;\Omega)\times (-\R_+).
    \]
Therefore,
    \[
    \partial \delta_\Omega(\infty) = \bigl\{u\in\R^n\mid (u,-1)\in N(\infty_I;\epi\delta_\Omega)\bigr\} = N(\infty;\Omega).
\]
The statement for $\partial^{\infty} \delta_\Omega(\infty)$ follows in a similar way.
\end{example}

\subsection{Basic Properties}
We collect the basic properties of $\partial f(\infty)$ and $\partial^{\infty} f(\infty)$.

\begin{proposition} \label{prop: basic subdiff at infinity}
Let $f:\R^n \to \R \cup \{\infty\}$ be a proper, lower semicontinuous, convex function with $\dom f$ unbounded. Then:
\begin{enumerate}[\rm (i)]
    \item $\partial f(\infty)$ is a closed convex set, and $\partial^{\infty} f(\infty)$ is a closed convex cone;
    \item the multifunction $\partial f(\cdot)$ is closed at infinity in the sense that
    \[
    x_k \to \infty,\ \ u_k \in \partial f(x_k),\ \ u_k \to u \quad \Longrightarrow \quad u \in \partial f(\infty);
    \]
    \item $\displaystyle \partial^M f(\infty) = \Limsup_{x \to \infty} \partial f(x) \subset \partial f(\infty)$.
\end{enumerate}
\end{proposition}

\begin{proof}
    (i)--(ii) follow from the corresponding closedness and convexity properties of the normal cone at infinity (see Proposition~\ref{prop: basic properties}). 
Here, since $I=\{1,\ldots,n\}$, only $x$ escapes to infinity, while $r$ may remain bounded.

(iii) The first equality follows from Proposition 4.4 in \cite{tung23b}. For the latter inclusion, take any $u \in \Limsup_{x \to \infty} \partial f(x)$. Then there exist sequences $\{x_k\}\subset\dom f$ and $\{u_k\}\subset\R^n$ with $\|x_k\|\to\infty$ and $u_k\to u$ such that $u_k\in\partial f(x_k)$ for all $k\in\N$. In particular,
    \[
    (u_k,-1)\in N\big((x_k,f(x_k));\epi f\big).
    \]
Passing to the limit and using $\|\pi(x_k,f(x_k))\|=\|x_k\|\to\infty$, we obtain
    \[
    (u,-1)\in \Limsup_{z\to\infty} N(z;\epi f) \subset N(\infty_I;\epi f).
    \]
Hence $u\in\partial f(\infty)$.
\end{proof}

\begin{proposition} \label{prop: singular and subdifferential}
Let $f:\R^n \to \R \cup \{\infty\}$ be a proper, lower semicontinuous, convex function with $\dom f$ unbounded. Consider the sets
    \[
    A:=\Limsup_{x \to \infty} \partial^\infty f(x), \quad
    B:=\partial^\infty f(\infty), \quad
    C:=\Limsup_{t \searrow 0} t\,\partial f(\infty), \quad
    D:=\Limsup_{x \to \infty, \;t \searrow 0} t\,\partial f(x).
    \]
We have the following relations:
    \[
    A \subset D \subset B \quad \text{and} \quad C \subset B.
    \]
If, in addition, $\partial f(\infty) \neq \emptyset$, then $C = B$.
\end{proposition}

\begin{proof}
\emph{Step 1: $A\subset D$.} Let $u \in A$. Then there exist sequences $\{x_k\}\subset\dom f$ and $\{u_k\}\subset\R^n$ such that $\|x_k\|\to\infty$, $u_k\to u$, and $u_k\in\partial^\infty f(x_k)$ for all $k\in\N$. If $\partial f(x_k)=\emptyset$ for infinitely many $k$, then by \cite[Lemma~2.6]{tung23b} there exist $x_k'\in\dom f$, $u_k'\in\partial f(x_k')$, and $t_k'\searrow 0$ such that $\|x_k'-x_k\|\to 0$ and $t_k'u_k'\to u$. Since $\|x_k'-x_k\|\to 0$ gives us $\|x_k'\|\to\infty$, we have $u\in D$. We may therefore assume that $\partial f(x_k)\neq\emptyset$ for all $k$. For each $k$, pick $p_k \in \partial f(x_k)$. Then
    \[
    (u_k,0)\in N\big((x_k,f(x_k));\epi f\big),\qquad (p_k,-1) \in N\big((x_k,f(x_k));\epi f\big).
    \]
Since $N\big((x_k,f(x_k));\epi f\big)$ is a convex cone, setting $t_k := \min\left\{\frac{1}{k}, \frac{1}{k(1 + \|p_k\|)}\right\} > 0$ gives us
    \[
    \Big(\frac{u_k}{t_k}+p_k,\,-1\Big) = \frac{1}{t_k}(u_k,0)+(p_k,-1) \in N\big((x_k,f(x_k));\epi f\big).
    \]
Hence $w_k := u_k/t_k+p_k \in \partial f(x_k)$. Moreover, we have
    \[
    \|t_kw_k-u\| \le \|t_k w_k-u_k\|+\|u_k-u\| = t_k\|p_k\|+\|u_k-u\| \le \frac{1}{k}+\|u_k-u\| \to 0,
    \]
which implies that $t_k w_k \to u$. So $u\in D$, and thus $A \subset D$.

\smallskip
\emph{Step 2: $D \subset B$.} Take $u\in D$. Then there exist sequences $\{x_k\}\subset\dom f$, $\{t_k\}\subset\R_+$, and $\{u_k\}\subset\R^n$ such that $\|x_k\|\to\infty$, $t_k\searrow 0$, $u_k\in\partial f(x_k)$ for all $k\in\N$, and $t_k u_k\to u$. Since $(u_k,-1)\in N\big((x_k,f(x_k));\epi f\big)$, we have $(t_k u_k,-t_k)\in N\big((x_k,f(x_k));\epi f\big)$. Passing to the limit, we get $(u,0)\in N(\infty_I;\epi f)$. Equivalently, $u\in B$.

\smallskip
\emph{Step 3: $C \subset B$.} Take $u\in C$. Then there exist sequences $\{t_k\}\subset\R_+$ and $\{v_k\}\subset\partial f(\infty)$ with $t_k\searrow 0$ and $t_k v_k\to u$. Since $(v_k,-1)\in N(\infty_I;\epi f)$, we have $(t_k v_k,-t_k)\in N(\infty_I;\epi f)$. Passing to the limit, we get $(u,0)\in N(\infty_I;\epi f)$, so $u\in B$.

\emph{Step 4: $B \subset C$.} Assume that $\partial f(\infty)\neq\emptyset$ and take $u\in B$. Then $(u,0)\in N(\infty_I;\epi f)$. Choose $w\in \partial f(\infty)$. For every $t>0$, we have
\[
\Big(\frac{u}{t}+w,-1\Big)=\frac1t(u,0)+(w,-1)\in N(\infty_I;\epi f),
\]
so $v:=u/t+w\in\partial f(\infty)$. Multiplying by $t$ gives us $tv=u+tw\to u$ as $t\searrow 0$, hence $u\in C$ and $B\subset C$.
\end{proof}

We close this subsection with two examples showing that the inclusions in Proposition~\ref{prop: singular and subdifferential} can be strict. In general, there is no inclusion relation between $C$ and $A$, nor between $C$ and $D$.

\begin{example}
Consider a function $f:\R\to\R\cup\{\infty\}$ defined by $f(x)=x^2$ if $x\ge 0$, and $f(x)=\infty$ if $x<0$.
Then
    \[
    A=\Limsup_{|x| \to \infty}\partial^\infty f(x)=\{0\},
    \qquad
    D=\Limsup_{|x| \to \infty,\ t\searrow0} t\,\partial f(x)=\R_+.
    \]
Direct computation gives us
    \[
    N(\infty_I;\epi f)=\R_+\times\{0\},
    \]
and hence
    \[
    B=\partial^\infty f(\infty)=\R_+,
    \qquad
    C=\Limsup_{t\searrow0} t\,\partial f(\infty)=\varnothing.
    \]
In particular, \(B\neq C\), \(A\not\subset C\), and \(D\not\subset C\).
\end{example}

\begin{example}
Consider a set $\Omega:=\{(x_1,x_2)\in\R^2\mid x_2\ge e^{x_1}\}$ and a function $f:=\delta_\Omega$.
Then
    \[
    B=\partial^\infty f(\infty)=\R_+\times(-\R_+),
    \qquad
    C=\Limsup_{t\searrow0} t\,\partial f(\infty)
    =\R_+\times(-\R_+).
    \]
On the other hand,
    \[
    A=\Limsup_{|x|\to\infty}\partial^\infty f(x)
    =
    (\R_+\times\{0\})\cup(\{0\}\times(-\R_+)),
    \]
and
    \[
    D=\Limsup_{|x|\to\infty,\ t\searrow0} t\,\partial f(x)=A.
    \]
Thus \(C\not\subset A\) and \(C\not\subset D\).
\end{example}

\subsection{Geometric Characterizations}
We now characterize $\partial^\infty f(\infty)$ and $\partial f(\infty)$ in terms of classical convex-analytic objects: the normal cone $N(\infty;\dom f)$, the domain of the recession function $\dom f^\infty$, and the domain of the conjugate $\dom f^*$.

\begin{proposition} \label{prop: geometry of singular}
Let $f:\R^n \to \R \cup \{\infty\}$ be a proper, lower semicontinuous, convex function with $\dom f$ unbounded. Then
    \[
    N(\infty;\dom f) \subset \partial^\infty f(\infty) \subset (\dom f^\infty)^\circ.
    \]
Moreover, the first inclusion is an equality provided the following condition is satisfied:
    \begin{equation} \label{eq: dom to epi}
    \;\ \text{for all}~ w \in T(\infty;\dom f),\ \exists r \in \R\ \text{such that}\ (w,r)\in T(\infty_I;\epi f).
    \end{equation}
\end{proposition}

\begin{proof}
Since $\epi f\subset \dom f\times \R$, we have
    \[
    T(\infty_I;\epi f)\subset T(\infty_I;\dom f\times \R).
    \]
Moreover, for every $(x,y)\in \dom f\times \R$, we have $T\big((x,y);\dom f\times \R\big)=T(x;\dom f)\times \R$. Hence
    \[
    T(\infty_I;\dom f\times \R)=T(\infty;\dom f)\times \R.
    \]
Therefore,
    \[
    T(\infty_I;\epi f)\subset T(\infty;\dom f)\times \R.
    \]
Taking polars and using Proposition~\ref{prop: basic properties}(ii), we obtain
    \[
    N(\infty;\dom f)\times\{0\} = \big(T(\infty;\dom f)\times\R\big)^\circ \subset N(\infty_I;\epi f).
    \]
Thus $N(\infty;\dom f)\subset \partial^\infty f(\infty)$.

For the second inclusion, since $\epi f$ is a closed and convex set, it follows from \cite[Theorem~8.5]{rockafellar-convex} that $(\epi f)^\infty=\epi f^\infty$. Applying Proposition~\ref{prop: recession inclusion} to $\Omega=\epi f$, we have $N(\infty_I;\epi f)\subset (\epi f^\infty)^\circ$. A direct computation yields that
    \begin{align*}
    \partial^\infty f(\infty) &= \{u\in\R^n\mid (u,0)\in N(\infty_I;\epi f)\} \\
    &\subset \{u\in\R^n\mid (u,0)\in (\epi f^\infty)^\circ\} \\
    &= \{u\in\R^n\mid \langle u, w\rangle \le 0 \text{ for all } w\in\dom f^\infty\} = (\dom f^\infty)^\circ.
    \end{align*}

For the equality of the first inclusion, assume that \eqref{eq: dom to epi} is satisfied. Take any $u\in \partial^\infty f(\infty)$. Then $(u,0) \in N(\infty_I;\epi f)$, and
    \[
    \langle u,w\rangle \le 0 \quad \;\ \text{for all}~ (w,r)\in T(\infty_I;\epi f).
    \]
Now fix an element $w \in T(\infty;\dom f)$. By \eqref{eq: dom to epi}, there exists $r \in \R$ such that $(w,r)\in T(\infty_I;\epi f)$. Hence $\langle u,w \rangle \le 0$, and therefore $u \in N(\infty;\dom f)$.
\end{proof}

Condition~\eqref{eq: dom to epi} is essential for the equality $\partial^\infty f(\infty)=N(\infty;\dom f)$; the next example shows that the equality can fail when this condition is not satisfied.

\begin{example} \label{ex: exp dom-to-epi fails}
Consider a function $f:\R\to\R$ defined by $f(x)=e^x$. Then $\dom f=\R$, so
    \[
    T(\infty;\dom f)=\R,\qquad N(\infty;\dom f)=\{0\}.
    \]
A direct computation gives us
    \[
    T(\infty_I;\epi f)=\R_-\times\R_+,
    \qquad
    N(\infty_I;\epi f)=\R_+\times\R_-.
    \]
Hence
    \[
    \partial^\infty f(\infty)=\R_+\supsetneq N(\infty;\dom f)=\{0\}, \quad \partial f(\infty)=\R_+.
    \]
Condition~\eqref{eq: dom to epi} fails here because if $w>0$, then $w\in T(\infty;\dom f)$, but no $r\in\R$ satisfies $(w,r)\in T(\infty_I;\epi f)$.
\end{example}

\begin{theorem} \label{prop: geometry of subdifferential}
Let $f:\R^n\to\R\cup\{\infty\}$ be a proper, lower semicontinuous, convex function with $\dom f$ unbounded. Then
    \[
    \partial f(\infty) \subset \cl(\dom f^*).
    \]
Moreover, if $f:\R^n\to\R$ is sublinear, then $T(\infty_I;\epi f)=\epi f$, the set $\dom f^*$ is closed, and equality holds:
    \[
    \partial f(\infty)=\partial f(0)=\dom f^*.
    \]
\end{theorem}

\begin{proof}
We first prove the inclusion. Since $\epi f$ is a closed and convex set, it follows from \cite[Theorem~8.5]{rockafellar-convex} that $(\epi f)^\infty=\epi f^\infty$. Applying Proposition~\ref{prop: recession inclusion} to $\Omega=\epi f$, we have $N(\infty_I;\epi f)\subset (\epi f^\infty)^\circ$. A direct computation yields that
\begin{align*}
    \partial f(\infty) &= \{u\in\R^n\mid (u,-1)\in N(\infty_I;\epi f)\} \\
    &\subset \{u\in\R^n\mid (u,-1)\in (\epi f^\infty)^\circ\} \\
    &= \{u\in\R^n\mid \langle u,w\rangle\le f^\infty(w) \text{ for all } w\in\R^n\} = \cl(\dom f^*),
\end{align*}
where the last equality follows from the fact that $f^\infty=\sigma_{\dom f^*}$.

For the moreover part, suppose $f:\R^n\to\R$ is sublinear. Then $f^\infty=f$ and $\epi f$ is a closed convex cone. By \cite[Theorem~13.2 and Corollary~13.2.1]{rockafellar-convex}, we have
\[
    f = \sigma_{\partial f(0)} \quad\text{and}\quad f^* = \delta_{\partial f(0)}.
\]
So $\dom f^* = \partial f(0)$ is a closed convex set; in particular, $\cl(\dom f^*) = \dom f^*$, and the inclusion proved above reads $\partial f(\infty)\subset \dom f^*$. Hence, it suffices to prove that $\dom f^*\subset\partial f(\infty)$. To this end, we will show that $T(\infty_I;\epi f)\subset\epi f$. Let $(u,v) \in T(\infty_I;\epi f)$ and suppose to the contrary that $v<f(u)$. Using the presentation $f = \sigma_{\partial f(0)}$, we consider the following two cases.
\begin{itemize}
    \item If $u \neq 0$, we choose $\bar c \in \partial f(0)$ such that $\langle \bar c, u \rangle=f(u)$, and consider the closed affine half-space
    \[
    \Omega := \{(x,r) \in \R^n \times \R \mid \langle \bar c,x \rangle \le r\}.
    \]
    For $z_k := k(u,f(u)) \in \epi f$, we have $\|\pi(z_k)\| \to \infty$ and $z_k \in \bd\ \Omega$. Since $(u,v) \in T(\infty_I;\epi f)$, there exists a sequence $\{(u_k,v_k)\}\subset\R^n\times\R$ such that $(u_k,v_k) \to (u,v)$ and $(u_k,v_k) \in T(z_k;\epi f)$ for all $k$. Because $\epi f \subset \Omega$ and $z_k \in \bd\ \Omega$,
    \[
    T(z_k;\epi f) \subset T(z_k;\Omega) = \Omega.
    \]
    Thus $\langle \bar c,u_k \rangle \le v_k$ for all $k$. Passing to the limit yields $f(u) = \langle \bar c,u \rangle \le v$, a contradiction.

    \item If $u=0$, then $f(u)=f(0)=0$ and $v<0$. Choose any $u' \neq 0$ and $c' \in \partial f(0)$ such that $f(u') = \langle c', u'\rangle$. Repeating the above argument with $c'$ and $z'_k := k(u',f(u'))$ gives us $0 = \langle c',0 \rangle \le v$, again a contradiction.
\end{itemize}
Therefore $v\ge f(u)$, and hence $(u,v)\in \epi f$. This proves $T(\infty_I;\epi f)\subset \epi f$.

By polarity, we get
\begin{align*}
    \partial f(\infty) &\supset \{u\in\R^n\mid \langle u, w\rangle\le r \text{ for all } (w,r)\in\epi f\} \\
    &\supset \{u\in\R^n\mid \langle u,w\rangle\le f(w) \text{ for all } w\in\R^n\} = \partial f(0).
\end{align*}
Together with $\dom f^*=\partial f(0)$ and the inclusion from the first part, this gives us the desired equalities $\partial f(\infty)=\partial f(0)=\dom f^*$. The polar equality $T(\infty_I;\epi f)=\epi f$ then follows.
\end{proof}

When $f$ is polyhedral, the inclusion $\partial f(\infty)\subset \cl(\dom f^*)$ takes a more concrete form.

\begin{proposition} \label{cor: polyhedral subdiff}
Let $f:\R^n\to\R\cup\{\infty\}$ be a polyhedral, proper, lower semicontinuous, convex function with $\dom f$ unbounded. Then $\dom f^*$ is a polyhedral convex set, hence closed, and
\[
    \partial f(\infty)\subset \dom f^*.
\]
If, moreover, $\dom f=\R^n$, then the inclusion is an equality:
\[
    \partial f(\infty) = \dom f^*.
\]
\end{proposition}

\begin{proof}
By \cite[Theorem~19.2]{rockafellar-convex}, we have that $f^*$ is polyhedral. In particular, $\dom f^*$ is a polyhedral convex set, and hence closed. So the inclusion $\partial f(\infty) \subset \dom f^*$ follows from Theorem~\ref{prop: geometry of subdifferential}.

For the equality, assume that $\dom f=\R^n$ and $f$ admits the presentation $f(x)=\max_{j=1,\ldots,m}(\langle a_j,x\rangle-b_j)$, where $a_1,\ldots,a_m \in \R^n \setminus \{0\}$ and $b_1,\ldots,b_m \in \R$. It is clear that $f^\infty(d)=\max_{j=1,\ldots,m}\langle a_j,d\rangle$ for every $d\in\R^n$. By \cite[Theorem~13.3]{rockafellar-convex}, we have $f^\infty=\sigma_{\dom f^*}$, and so
\[
    \dom f^* = \conv\{a_j\mid j=1,\ldots,m\}.
\]
On the other hand, since $\epi f$ is a polyhedral convex set, applying Theorem~\ref{prop: polyhedral cones at infinity} to $\epi f$ gives us
\[
    \partial f(\infty) = \conv\{a_j\mid j\in J^\infty_I(\epi f)\},
\]
where $J^\infty_I(\epi f)$ consists of those $j$ whose active region $R_j:=\{x\in\R^n\mid f(x)=\langle a_j,x\rangle-b_j\}$ is unbounded (cf.\ Definition~\ref{def: asymptotic active set}).

It remains to prove the reverse inclusion $\dom f^*\subset\partial f(\infty)$, that is,
\[
    \conv\{a_j\mid j=1,\ldots,m\} \subset \conv\{a_j \mid j\in J^\infty_I(\epi f)\}.
\]
Since both sides are compact convex sets, this is equivalent to the inequality of their support functions,
\[
    f^\infty(d)\;\le\;\max_{j\in J^\infty_I(\epi f)}\langle a_j,d\rangle\qquad \text{for all } d\in\R^n.
\]
To prove this inequality, fix any $d\in\R^n$ and $x_0\in\R^n$. For every $t\ge 0$,
\[
    f(x_0+td)=\max_{j=1,\ldots,m}\big(\langle a_j,x_0\rangle-b_j+t\langle a_j,d\rangle\big),
\]
a maximum of $m$ affine functions of $t$. Hence there exist an index $j^*$ and a number $T\ge 0$ such that
\[
    f(x_0+td)=\langle a_{j^*},x_0\rangle-b_{j^*}+t\langle a_{j^*},d\rangle \qquad \text{for all } t\ge T.
\]
In particular, $x_0+td\in R_{j^*}$ for all $t\ge T$, so $R_{j^*}$ is unbounded and $j^*\in J^\infty_I(\epi f)$. Dividing by $t$ and letting $t\to\infty$ yields
\[
    f^\infty(d)=\lim_{t\to\infty}\frac{f(x_0+td)}{t}=\langle a_{j^*},d\rangle\le\max_{j\in J^\infty_I(\epi f)}\langle a_j,d\rangle,
\]
as required. Together with the inclusion $\partial f(\infty)\subset\dom f^*$ established above, this yields $\partial f(\infty)=\dom f^*$.
\end{proof}

The equality part of Proposition~\ref{cor: polyhedral subdiff} requires $\dom f=\R^n$. Without this condition, the inclusion may be strict. For instance, $f(x_1,x_2):=x_1$ on $\dom f:=\R_+\times[0,1]$ is polyhedral with $\dom f^*=\{(u_1,u_2)\mid u_1\le 1\}$, whereas $\partial f(\infty)=\{1\}\times\R\subsetneq \dom f^*$.

\begin{remark}
Theorem~\ref{prop: geometry of subdifferential} shows that, for a sublinear function, $\partial f(\infty)$ coincides with the usual convex subdifferential at $0$. The same conclusion may fail for the limiting subdifferential $\partial^M f(\infty)$. Indeed, consider the sublinear function $f(x)=|x|$ for $x\in\R$. Theorem~\ref{prop: geometry of subdifferential} gives us
    \[
    \partial f(\infty)=\partial f(0)=\dom f^*=[-1,1].
    \]
In contrast,
    \[
    \partial^M f(0)=\partial f(0)=[-1,1],
\qquad
    \partial^M f(\infty)=\{-1,1\}.
    \]
\end{remark}

\begin{remark}
\begin{enumerate}[\rm (i)]
    \item In the second part of Theorem~\ref{prop: geometry of subdifferential}, the function $f$ is required to be finite-valued. The conclusion may fail for extended-valued sublinear functions. For instance, let $f=\delta_{\R_+}$ on $\R$. Then $f$ is a proper, lower semicontinuous, convex, extended-valued sublinear function, and $\dom f=\R_+$ is unbounded. We have
    \[
    \partial f(0)=N_{\R_+}(0)=\R_-,\qquad f^*=\delta_{\R_-},\qquad \dom f^*=\R_-.
    \]
    On the other hand, $\epi f=\R_+\times\R_+$, so
    \[
    T(\infty_I;\epi f)=\R\times\R_+,\qquad N(\infty_I;\epi f)=(\R\times\R_+)^\circ=\{0\}\times\R_-.
    \]
    Therefore,
    \[
    \partial f(\infty) = \{u \in\R \mid (u,-1)\in \{0\}\times \R_-\} = \{0\} \neq \R_- = \partial f(0)=\dom f^*.
    \]

    \item The equality $\partial f(\infty)=\dom f^*$ needs more than properness, lower semicontinuity, and convexity. For instance, if $f(x)=x^2$ for $x\in\R$, then by a direct computation (cf.\ Example~\ref{ex: exp counterexample})
    $
    N(\infty_I;\epi f)=\R\times\{0\},
    $
and hence
    \[
    \partial f(\infty)=\varnothing.
    \]
On the other hand,
    \[
    \partial f(0)=\{0\},
    \qquad
    \dom f^*=\R.
    \]
Only the inclusion $\partial f(\infty)\subset\cl(\dom f^*)$ survives. Thus sublinearity is a sufficient condition, but it is not necessary for the equality. By Proposition~\ref{cor: polyhedral subdiff}, polyhedrality together with $\dom f=\R^n$ is another sufficient condition.
\end{enumerate}
\end{remark}

\begin{remark} \label{rem: sublinear facts} 
	Let $f:\R^n\to\R$ be a sublinear function. Then $\partial^\infty f(\infty)=\{0\}$ and $\partial f(\infty)$ is bounded. In fact, since $\dom f=\R^n$, we have
	\[
	T(\infty;\dom f)=\R^n \quad \text{and} \quad N(\infty;\dom f)=\{0\}.
	\] 
By Theorem~\ref{prop: geometry of subdifferential}, we have proven that $T(\infty_I;\epi f)=\epi f$. So for any $w\in T(\infty;\dom f)=\R^n$, choosing $r:=f(w)$ gives us $(w,r)\in\epi f=T(\infty_I;\epi f)$, which verifies condition~\eqref{eq: dom to epi}. Applying Proposition~\ref{prop: geometry of singular}, we obtain
    \[
    \partial^\infty f(\infty)=N(\infty;\dom f)=\{0\}.
    \]
Furthermore, combining this with Proposition~\ref{prop: singular and subdifferential}, we conclude that
    \[
    \Limsup_{t \searrow 0} t\,\partial f(\infty) = \{0\}.
    \]
    In particular, $\partial f(\infty)$ is bounded.
\end{remark}

We conclude this subsection with a few examples illustrating Theorem~\ref{prop: geometry of subdifferential} for several classes of sublinear functions.

\begin{example} \label{ex: support function}
Let \(\Omega\subset \R^n\) be a nonempty bounded set, and consider its support function $\sigma_\Omega$. Since $\Omega$ is bounded, $\sigma_\Omega$ is finite-valued on $\R^n$. Moreover, $\sigma_\Omega$ is sublinear, and it is well known that
    \[
    (\sigma_\Omega)^*(u) = \delta_{\cl\conv\ \Omega}(u).
    \]
Hence Theorem~\ref{prop: geometry of subdifferential} gives us
    \[
    \partial \sigma_\Omega(\infty) = \dom (\sigma_{\Omega})^* = \cl\conv\ \Omega.
    \]
\end{example}

\begin{example}
Let \(\{a_\lambda\}_{\lambda\in\Lambda}\subset \R^n\) be a bounded family and define
    \[
    f(x):=\sup_{\lambda\in\Lambda}\langle a_\lambda,x\rangle, \quad x\in\R^n.
    \]
Set $A:=\{a_\lambda\mid \lambda\in\Lambda\}$. Then $f=\sigma_A$. By Example~\ref{ex: support function}, we obtain
    \[
    \partial f(\infty) = \cl\conv\ A.
    \]
\end{example}

\begin{example} \label{ex: distance cone}
Let $K\subset\R^n$ be a nonempty closed convex cone, and define $d_K:\R^n\to\R$ by
    \[
    d_K(x):={\rm d}(x; K),\qquad x\in\R^n.
    \]
By Lemma~\ref{lemma: moreau} (with $C=K$), we have
    \[
    d_K(x)=\sup\bigl\{\langle u,x\rangle\mid u\in K^\circ,\ \|u\|\le 1\bigr\}
    =\sigma_{K^\circ \cap ~\cl~{\mathbb{B}}}(x).
    \]
As $K^\circ \cap ~\cl~{\mathbb{B}}$ is bounded, closed, and convex, Example~\ref{ex: support function} yields
    \[
    \partial d_K(\infty)=K^\circ \cap ~\cl~{\mathbb{B}}.
    \]
\end{example}

\begin{example}
Let \(C\subset \R^n\) be a nonempty closed convex absorbing set (that is, for every $x \in \R^n$ there exists $t>0$ such that $x \in tC$), with $0\in C$. Define the Minkowski functional
    \[
    \mu_C(x):=\inf\{t>0\mid x\in tC\}, \qquad x\in\R^n.
    \]
Since $C$ is absorbing, $\mu_C$ is finite-valued on $\R^n$. Moreover, because $C$ is convex and contains the origin, $\mu_C$ is sublinear. It is well known that
    \[
    (\mu_C)^*(u)=\delta_{C^\circ}(u).
    \]
Consequently, Theorem~\ref{prop: geometry of subdifferential} yields
    \[
    \partial \mu_C(\infty) = \dom(\mu_C)^* = C^\circ.
    \]
\end{example}

\subsection{Calculus Rules}

We start this section with a sum rule. The proof relies on a lifted intersection representation of $\epi h$ (where $h=f+g$), combined with the intersection rule for normal cones at infinity.

\begin{proposition} \label{prop: sum rule same-variable}
Let $f,g:\R^n\to\R\cup\{\infty\}$ be proper, lower semicontinuous, convex functions, and set $h:=f+g$. Assume that $\dom h$ is unbounded and that
\begin{equation} \label{eq: singular separation}
    \partial^\infty f(\infty) \cap \bigl(-\partial^\infty g(\infty)\bigr) = \{0\}.
\end{equation}
Then
\[
    \partial h(\infty) \subset \partial f(\infty)+\partial g(\infty).
\]
\end{proposition}

\begin{proof}
Set
\begin{align*}
   \Omega_1 := \{(x,r,s) \in \R^n\times\R\times\R \mid r\ge f(x)\}, \\
   \Omega_2 := \{(x,r,s) \in \R^n\times\R\times\R \mid s\ge g(x)\}.
\end{align*}
Then $\Omega_1 \cap \Omega_2 = \{(x,r,s)\mid r\ge f(x),\ s\ge g(x)\}$. Consider the linear mapping
    \[
    A:\R^{n+2}\to\R^{n+1},\qquad
    A(x,r,s):=(x,r+s).
    \]
A direct computation gives $A(\Omega_1\cap\Omega_2)=\epi h$. Let $I=\{1,\ldots,n\}$ select the $x$-coordinates of each space, so that $\pi_I(x,r,s)=\pi_I(A(x,r,s))=x$. Then the two hypotheses of Proposition~\ref{prop: normal linear image} are satisfied: condition~\eqref{eq: compatibility infinity} holds, and $\pi_I(A(\Omega_1\cap\Omega_2))=\dom h$ is unbounded by hypothesis. Now take any $u\in\partial h(\infty)$, so that $(u,-1)\in N(\infty_I;\epi h)$. By Proposition \ref{prop: normal linear image}(ii),
    \begin{equation} \label{eq: Atop normal}
          A^\top(u,-1) = (u,-1,-1) \in N(\infty_I;\Omega_1\cap\Omega_2).  
    \end{equation}
Moreover
\begin{align*}
    N(\infty_I;\Omega_1) = \{(u,\alpha,0) \mid (u,\alpha) \in N(\infty_I;\epi f)\}, \\
    N(\infty_I;\Omega_2) =\{(u,0,\beta) \mid (u,\beta) \in N(\infty_I;\epi g)\}.
\end{align*}
Hence
    \[
    N(\infty_I;\Omega_1) \cap \bigl(-N(\infty_I;\Omega_2)\bigr)
    = \{(u,0,0)\mid (u,0)\in N(\infty_I;\epi f),\ (-u,0)\in N(\infty_I;\epi g)\}.
    \]
Therefore, the assumption \eqref{eq: singular separation} is equivalent to \eqref{eq: normal separation at infinity} in Proposition~\ref{prop: normal intersection rule}, and so
    \begin{equation} \label{eq: normal intersection}
    N(\infty_I;\Omega_1\cap\Omega_2) \subset N(\infty_I;\Omega_1)+N(\infty_I;\Omega_2).
    \end{equation}
Combining \eqref{eq: Atop normal} and \eqref{eq: normal intersection}, we can write
    \[
    (u,-1,-1)=(a,-1,0)+(b,0,-1),
    \]
with
    \[
    (a,-1,0) \in N(\infty_I;\Omega_1),
    \qquad
    (b,0,-1) \in N(\infty_I;\Omega_2).
    \]
By the above presentations of $N(\infty_I;\Omega_1)$ and $N(\infty_I;\Omega_2)$, we obtain $(a,-1)\in N(\infty_I;\epi f)$ and $(b,-1)\in N(\infty_I;\epi g)$, that is,
    \[
     a\in \partial f(\infty),\qquad b\in \partial g(\infty).
    \]
Hence $u=a+b \in \partial f(\infty)+\partial g(\infty)$, which completes the proof.
\end{proof}

\begin{remark} \label{rem: sum polyhedral}
If $f$ and $g$ are polyhedral, then $h=f+g$ is polyhedral with $\dom h$ a polyhedral convex set. Furthermore, it follows from~\cite[Theorem~20.1]{rockafellar-convex} that $(f+g)^*=f^*\square g^*$, so $\dom h^*=\dom f^*+\dom g^*$. Proposition~\ref{cor: polyhedral subdiff} then yields the inclusion $\partial h(\infty)\subset\dom f^*+\dom g^*$ without requiring \eqref{eq: singular separation}.
\end{remark}

In particular, if one of the two functions is sublinear and finite-valued, then \eqref{eq: singular separation} holds automatically by Remark~\ref{rem: sublinear facts}, and the conclusion of Proposition~\ref{prop: sum rule same-variable} needs no further qualification.

\begin{corollary}
Let $f: \R^n \to \R\cup\{\infty\}$ be a proper, lower semicontinuous, convex function, and $g: \R^n \to \R$ be a sublinear function. Assume that $\dom (f+g)$ is unbounded. Then
    \[
    \partial (f+g)(\infty)\subset \partial f(\infty)+\partial g(\infty).
    \]
\end{corollary}

\begin{proposition} \label{prop: chain rule linear}
Let $A:\R^n\to\R^n$ be invertible, and let $f:\R^n\to\R\cup\{\infty\}$ be a proper, lower semicontinuous, convex function with $\dom f$ unbounded. Define
\[
h:\R^n\to\R\cup\{\infty\},\qquad h(x):=f(Ax).
\]
Then
\[
\partial h(\infty)=A^\top[\partial f(\infty)].
\]
\end{proposition}

\begin{proof}
Define $B:\R^{n+1}\to\R^{n+1}$ by $B(x,r):=(Ax,r)$, and let $\pi:\R^{n+1}\to\R^n$ be the projection $(x,r)\mapsto x$ (so $I=\{1,\ldots,n\}$). Then $B$ is invertible, $B^{-1}(y,r)=(A^{-1}y,r)$, and $B(\epi h)=\epi f$. Since $A$ is invertible, $\|Ax\|\to\infty$ if and only if $\|x\|\to\infty$, so both \eqref{eq: compatibility infinity} and \eqref{eq: converse compatibility infinity} hold for $B$ on $\epi h$. Moreover, $\pi(B(\epi h))=\pi(\epi f)=\dom f$ is unbounded by hypothesis. By Proposition~\ref{prop: normal linear image}(iii),
\[
B^\top N(\infty_I;\epi f)=N(\infty_I;\epi h).
\]
Let $u\in\R^n$. Since $B^\top(u,-1)=(A^\top u,-1)$, we have
\[
u\in\partial f(\infty)\iff(u,-1)\in N(\infty_I;\epi f)\iff(A^\top u,-1)\in N(\infty_I;\epi h)\iff A^\top u\in\partial h(\infty).
\]
This yields $\partial h(\infty)=A^\top[\partial f(\infty)]$.
\end{proof}

\begin{proposition} \label{prop: max rule}
Let $f,g:\R^n\to\R\cup\{\infty\}$ be proper, lower semicontinuous, convex functions and set
\[
    h(x):=\max\{f(x),g(x)\},\qquad x\in\R^n.
\]
Assume that $\dom h$ is unbounded and that \eqref{eq: singular separation} is satisfied. Then
\[
    \partial h(\infty)
    \subset
    \conv\bigl[\partial f(\infty)\cup\partial g(\infty)\bigr]
    \;\cup\; \bigl[\partial f(\infty)+\partial^\infty g(\infty)\bigr]
    \;\cup\; \bigl[\partial g(\infty)+\partial^\infty f(\infty)\bigr].
\]
\end{proposition}

\begin{proof}
Since $\epi h=\epi f\cap\epi g$, we apply Proposition~\ref{prop: normal intersection rule}(ii) to $\epi f$ and $\epi g$.

We first verify the condition~\eqref{eq: normal separation at infinity} required in Proposition~\ref{prop: normal intersection rule}. Set $e:=(0,1)\in \R^{n+1}$.
We claim that
\begin{equation} \label{eq: e in epi tangent}
e\in T(\infty_I;\epi f)
\quad\text{and}\quad
e\in T(\infty_I;\epi g).
\end{equation}
Indeed, if \((x,r)\in\epi f\) and \(t\ge 0\), then $(x,r)+te=(x,r+t)\in\epi f$. The same argument applies to \(\epi g\).

Now take $(u,\alpha)\in N(\infty_I;\epi f)\cap
\big[-N(\infty_I;\epi g)\big]$. On one hand, since $(u,\alpha)\in N(\infty_I;\epi f)$ and $e\in T(\infty_I;\epi f)$, we have $\alpha\le 0$. On the other hand, since $(u,\alpha)\in -N(\infty_I;\epi g)$ and $e\in T(\infty_I;\epi g)$, we obtain $\alpha\ge 0$. Hence $\alpha = 0$, and
\[
(u,0)\in N(\infty_I;\epi f),
\qquad
(-u,0)\in N(\infty_I;\epi g),
\]
which means
\[
u\in\partial^\infty f(\infty)
\quad\text{and}\quad
-u\in\partial^\infty g(\infty).
\]
By condition~\eqref{eq: singular separation}, we get \(u=0\). Thus
\[
N(\infty_I;\epi f)\cap
\big[-N(\infty_I;\epi g)\big]
=\{0\}.
\]
Consequently, Proposition~\ref{prop: normal intersection rule}(ii) gives
    \[
    N(\infty_I;\epi h)\subset N(\infty_I;\epi f)+N(\infty_I;\epi g).
    \]
Fix $u\in\partial h(\infty)$. By definition, we have $(u,-1)\in N(\infty_I;\epi h)$. Hence, there exist $(u_1,\alpha_1)\in N(\infty_I;\epi f)$ and $(u_2,\alpha_2)\in N(\infty_I;\epi g)$ such that
\[
    (u,-1)=(u_1,\alpha_1)+(u_2,\alpha_2).
\]
By a similar discussion as above, \eqref{eq: e in epi tangent} yields $\alpha_1\le 0$ and $\alpha_2\le 0$. Since $\alpha_1+\alpha_2=-1$, we distinguish three cases.

\begin{itemize}
    \item Case 1: $\alpha_1<0$ and $\alpha_2<0$. Set $\lambda_1:=-\alpha_1>0$ and $\lambda_2:=-\alpha_2>0$. Then $\lambda_1+\lambda_2=1$. Since normal cones at infinity are cones, we have
    \[
    \left(\frac{u_1}{\lambda_1},-1\right)\in N(\infty_I;\epi f),
    \qquad
    \left(\frac{u_2}{\lambda_2},-1\right)\in N(\infty_I;\epi g),
    \]
so $u_1/\lambda_1\in\partial f(\infty)$ and $u_2/\lambda_2\in\partial g(\infty)$. Hence
    \[
    u=\lambda_1\frac{u_1}{\lambda_1}+\lambda_2\frac{u_2}{\lambda_2}
    \in \conv\bigl(\partial f(\infty)\cup\partial g(\infty)\bigr).
    \]

    \item Case 2: $\alpha_1=0$ and $\alpha_2=-1$. Then $(u_1,0)\in N(\infty_I;\epi f)$ and $(u_2,-1)\in N(\infty_I;\epi g)$, so $u_1\in\partial^\infty f(\infty)$ and $u_2\in\partial g(\infty)$. Thus $u\in \partial g(\infty)+\partial^\infty f(\infty)$.

    \item Case 3: $\alpha_1=-1$ and $\alpha_2=0$. Similarly, $u\in \partial f(\infty)+\partial^\infty g(\infty)$.
\end{itemize}
Combining the three cases gives us the desired inclusion.
\end{proof}

We close this subsection with two further rules at infinity, for the infimal convolution and for the marginal function of a jointly convex parametric problem. The \emph{infimal convolution} of two functions $f,g:\R^n\to\R\cup\{\infty\}$ is defined by
\[
    (f\square g)(x) := \inf_{y\in\R^n}\bigl\{f(x-y)+g(y)\bigr\},\qquad x\in\R^n.
\]

\begin{proposition} \label{prop: inf-conv at infinity}
Let $f,g:\R^n\to\R\cup\{\infty\}$ be proper, lower semicontinuous, convex functions, and set $h:=f\square g$. Assume that $h$ is proper and lower semicontinuous with $\dom h$ unbounded. Then:
\begin{enumerate}[\rm (i)]
    \item $\partial h(\infty) \subset \cl(\dom f^*\cap\dom g^*)$;
    \item if $f$ and $g$ are sublinear and finite-valued, then $\partial h(\infty) = \partial f(\infty)\cap\partial g(\infty)$;
    \item if $f$ and $g$ are polyhedral, then $\partial h(\infty)\subset\dom f^*\cap\dom g^*$, with equality when $\dom f+\dom g=\R^n$.
\end{enumerate}
\end{proposition}

\begin{proof}
Since $f$ and $g$ are convex, so is $h$, and the Fenchel identity $h^*=f^*+g^*$~\cite[Theorem~16.4]{rockafellar-convex} gives
\[
    \dom h^*=\dom f^*\cap\dom g^*.
\]

(i) As $h$ is proper, lower semicontinuous, and convex with $\dom h$ unbounded, Theorem~\ref{prop: geometry of subdifferential} gives $\partial h(\infty)\subset\cl(\dom h^*)=\cl(\dom f^*\cap\dom g^*)$.

(ii) If $f$ and $g$ are sublinear and finite-valued, then $h$ is finite-valued sublinear, so the equality in Theorem~\ref{prop: geometry of subdifferential} gives us $\partial h(\infty)=\dom h^*=\dom f^*\cap\dom g^* = \partial f(\infty)\cap\partial g(\infty)$.

(iii) If $f$ and $g$ are polyhedral, then $h$ is polyhedral, so Proposition~\ref{cor: polyhedral subdiff} gives
\[
    \partial h(\infty)\subset\dom h^*=\dom f^*\cap\dom g^*,
\]
with equality when $\dom h=\dom f+\dom g=\R^n$.
\end{proof}

The \emph{marginal function} of a proper, lower semicontinuous, convex function $F:\R^n\times\R^p\to\R\cup\{\infty\}$ is $m(x):=\inf_{y\in\R^p}F(x,y)$. For a nonempty index set $I\subset\{1,\ldots,n+p\}$ with $\pi_I(\dom F)$ unbounded, set
\[
    \partial F(\infty_I) := \big\{z^*\in\R^{n+p}\mid (z^*,-1)\in N(\infty_I;\epi F)\big\},
\]
where $I$, read inside $\{1,\ldots,n+p+1\}$, refers only to the base variables of $\epi F$. This is the partial-escape form of Definition~\ref{def: subdiff infinity}(i).

\begin{proposition} \label{prop: marginal at infinity}
Assume that $m$ is proper and lower semicontinuous with $\dom m$ unbounded, and that the infimum is attained for every $x\in\dom m$. Then
\[
    \partial m(\infty) \subset \bigl\{u\in\R^n\mid (u,0)\in\partial F(\infty_I)\bigr\},
\]
where $I=\{1,\ldots,n\}\subset\{1,\ldots,n+p\}$.
\end{proposition}

\begin{proof}
Consider the linear projection
\[
    \pi:\R^{n+p+1}\to\R^{n+1},\qquad \pi(x,y,r) := (x,r).
\]
Under the attainment assumption, $\pi(\epi F)=\epi m$. Since $\pi_I(x,y,r)=x=\pi_I(\pi(x,y,r))$, condition~\eqref{eq: compatibility infinity} holds, and $\pi_I(\pi(\epi F))=\dom m$ is unbounded by hypothesis. By Proposition~\ref{prop: normal linear image}(ii),
\[
    \pi^\top N(\infty_I;\epi m) \subset N(\infty_I;\epi F).
\]
For $u\in\partial m(\infty)$, we have $(u,-1)\in N(\infty_I;\epi m)$, which gives us $\pi^\top(u,-1) = ((u,0),-1)\in N(\infty_I;\epi F)$. Hence $(u,0)\in\partial F(\infty_I)$.
\end{proof}

\begin{remark} \label{rem: marginal polyhedral}
If $F$ is polyhedral and the infimum in the definition of $m$ is attained for every $x\in\dom m$, then $m$ is polyhedral on $\R^n$ with $m^*(u)=F^*(u,0)$. Proposition~\ref{cor: polyhedral subdiff} yields
\[
    \partial m(\infty) \subset \dom m^* = \{u\in\R^n\mid F^*(u,0)<\infty\}.
\]
\end{remark}

\subsection{Clarke Directional Derivative at Infinity} \label{subsec: clarke deriv infinity}

Following the Clarke-type framework at infinity in~\cite{tung23a,tung23b}, we study the directional derivative at infinity in the convex setting.

\begin{definition} \label{def: clarke deriv infinity}
Let $f:\R^n\to\R\cup\{\infty\}$ be a proper, lower semicontinuous, convex function with $\dom f$ unbounded. The \emph{Clarke directional derivative of $f$ at infinity} in direction $d\in\R^n$ is defined by
\[
    f^\circ(\infty;d) := \limsup_{x\xrightarrow{\dom f}\infty,\ t\searrow 0}\, \frac{f(x+td)-f(x)}{t}.
\]
\end{definition}

For $d\in\R^n$, let $\Phi_f(d):=\inf\{\alpha\in\R\mid (d,\alpha)\in T(\infty_I;\epi f)\}$ denote the lower-boundary function of the tangent cone $T(\infty_I;\epi f)$, with the convention $\inf\emptyset=+\infty$.

\begin{proposition} \label{prop: clarke deriv epi}
Let $f:\R^n\to\R\cup\{\infty\}$ be a proper, lower semicontinuous, convex function with $\dom f$ unbounded. Then:
\begin{enumerate}[\rm (i)]
    \item $\Phi_f(d)\le f^\circ(\infty;d)$ for every $d\in\R^n$;
    \item if $f$ is Lipschitz at infinity, then the inequality in {\rm(i)} is an equality.
\end{enumerate}
\end{proposition}

\begin{proof}
(i) 
If $f^\circ(\infty;d)=+\infty$, there is nothing to prove. If $f^\circ(\infty;d)=-\infty$, the argument below applies to every $\alpha\in\R$, so that $(d,\alpha)\in T(\infty_I;\epi f)$ for all real $\alpha$ and hence $\Phi_f(d)=-\infty$, which again yields the inequality. We may therefore assume that $f^\circ(\infty;d)$ is finite.

Suppose $\alpha>f^\circ(\infty;d)$. By the definition of $\limsup$, there exist $\gamma>0$ and $\varepsilon>0$ such that
\[
    \frac{f(x+td)-f(x)}{t} \le \alpha \qquad\text{whenever } x\in\dom f,\ \|x\|\ge\gamma,\ 0<t\le\varepsilon.
\]
Fix any sequences $\{(x_k,r_k)\}\subset\epi f$ and $\{t_k\}\subset\R_+$ with $\|x_k\|\to\infty$ and $t_k\searrow 0$. For large $k$, $r_k\ge f(x_k)$ and the inequality above gives us $r_k+t_k\alpha\ge f(x_k+t_kd)$. So
\[
	(x_k,r_k) + t_k(d,\alpha) = (x_k+t_kd, r_k+t_k\alpha) \in \epi f \quad \text{for large} ~ k.
\]
Hence $(d,\alpha)\in T(\infty_I;\epi f)$, and thus $\Phi_f(d) \le \alpha$. Letting $\alpha\searrow f^\circ(\infty;d)$ gives us (i).

(ii) For the reverse inequality, take any $(d,\alpha)\in T(\infty_I;\epi f)$ and let $L,\gamma,\rho$ be the constants from Definition~\ref{def: lipschitz at infinity}. Fix arbitrary sequences $\{x_k\}\subset\dom f$ and $\{t_k\}\subset\R_+$ with $\|x_k\|\to\infty$ and $t_k\searrow 0$. There exist $(d_k,\alpha_k)\to(d,\alpha)$ with $(x_k,f(x_k))+t_k(d_k,\alpha_k)\in\epi f$ for all $k$. Equivalently, it follows that $\alpha_k \ge (f(x_k+t_kd_k)-f(x_k))/t_k$ for all $k$. For large $k$ we have $\|x_k\|\ge\gamma$, and both $x_k+t_kd$ and $x_k+t_kd_k$ lie in $B(x_k;\rho)$. The assumed Lipschitz condition then gives us
\[
   \frac{f(x_k+t_kd)-f(x_k)}{t_k} = \frac{f(x_k+t_kd_k)-f(x_k)}{t_k} + \frac{f(x_k+t_kd)-f(x_k+t_kd_k)}{t_k} \le \alpha_k + L\|d_k-d\|.
\]
Taking the supremum over admissible $(x_k,t_k)$, then passing to the limit gives us $f^\circ(\infty;d)\le\alpha$. Since $\alpha$ is arbitrary with $(d,\alpha)\in T(\infty_I;\epi f)$, the proof for (ii) is done.
\end{proof}

The next result identifies the lower-boundary function of the tangent cone $T(\infty_I;\epi f)$ with the support function of $\partial f(\infty)$.

\begin{proposition} \label{prop: clarke deriv support}
Let $f:\R^n\to\R\cup\{\infty\}$ be proper, lower semicontinuous, convex with $\dom f$ unbounded. Then
\[
    \Phi_f(d) =
    \begin{cases}
        \sigma_{\partial f(\infty)}(d), & d\in(\partial^\infty f(\infty))^\circ,\\[2pt]
        +\infty, & \text{otherwise,}
    \end{cases}
\]
with the convention $\sigma_\emptyset\equiv-\infty$.
\end{proposition}

\begin{proof}
 Fix $d\in\R^n$. As shown in Proposition \ref{prop: max rule}, we have $(0,1) \in T(\infty_I;\epi f)$. So every $(u,c) \in N(\infty_I;\epi f)$ has $c\le 0$. This gives us the following presentation
\[
    N(\infty_I;\epi f)=\{(u,0)\mid u\in\partial^\infty f(\infty)\}\cup\{t(v,-1)\mid t>0,\ v\in\partial f(\infty)\}.
\]
Hence $(d,\alpha) \in T(\infty_I;\epi f)$ iff $\langle u,d\rangle+c\alpha\le 0$ for every $(u,c) \in N(\infty_I;\epi f)$. By the presentation above, this holds iff $\langle u,d\rangle\le 0$ for all $u\in\partial^\infty f(\infty)$ and $\alpha\ge\langle v,d\rangle$ for all $v\in\partial f(\infty)$. Equivalently, it follows that $d\in(\partial^\infty f(\infty))^\circ$ and $\alpha\ge\sigma_{\partial f(\infty)}(d)$. Taking the infimum over such $\alpha$, we conclude that $\Phi_f(d)=\sigma_{\partial f(\infty)}(d)$ if $d\in(\partial^\infty f(\infty))^\circ$, and $\Phi_f(d)=+\infty$ otherwise.
\end{proof}

\begin{remark} \label{rem: empty subdiff convention}
When $\partial f(\infty)=\emptyset$, the convention $\sigma_\emptyset\equiv-\infty$ gives $\Phi_f(d)=-\infty$ for every $d\in(\partial^\infty f(\infty))^\circ$. Thus $\Phi_f$ is improper.
\end{remark}

\begin{corollary} \label{rem: classical subdiff at infinity}
Let $f:\R^n\to\R$ be Lipschitz at infinity with constant $L$. Then:
\begin{enumerate}[\rm (i)]
    \item $|f^\circ(\infty;d)|\le L\|d\|$ for all $d\in\R^n$;
    \item $\partial^\infty f(\infty)=\{0\}$, $\partial f(\infty)$ is a nonempty compact convex subset of $L\,\overline{\mathbb{B}}$, and
    \[
        f^\circ(\infty;d)=\max\bigl\{\langle v,d\rangle\mid v\in\partial f(\infty)\bigr\}\quad \text{for all } d\in\R^n.
    \]
\end{enumerate}
\end{corollary}
\begin{proof}
For (i), the case $d=0$ is immediate, so assume $d\ne 0$. Let $\gamma,\rho>0$ be the constants accompanying $L$ in Definition~\ref{def: lipschitz at infinity}. For $\|x\|\ge\gamma$ and $0<t<\rho/\|d\|$, both $x$ and $x+td$ lie in $B(x;\rho)$, so the Lipschitz estimate gives
\[
    \left|\frac{f(x+td)-f(x)}{t}\right|\le \frac{L\,\|td\|}{t}=L\|d\|.
\]
Taking the limsup yields that $|f^\circ(\infty;d)|\le L\|d\|$. 

For (ii), since $\dom f^\infty=\R^n$, Proposition~\ref{prop: geometry of singular} gives us $\partial^\infty f(\infty)\subset(\dom f^\infty)^\circ=\{0\}$, so $\partial^\infty f(\infty)=\{0\}$ and $\Phi_f=\sigma_{\partial f(\infty)}$ on $\R^n$. By Proposition~\ref{prop: clarke deriv epi}(ii), we have $\Phi_f=f^\circ(\infty;\cdot)$. Furthermore, it follows from part (i) that $\sigma_{\partial f(\infty)}=f^\circ(\infty;\cdot)$ is finite and bounded by $L\|\cdot\|$. So $\partial f(\infty)$ is a nonempty compact convex subset of $L\,\overline{\mathbb{B}}$ and the supremum is attained as a maximum.
\end{proof}

The next proposition compares $f^\circ(\infty;\cdot)$ with the recession function.

\begin{proposition} \label{prop: clarke deriv vs recession}
Let $f:\R^n\to\R\cup\{\infty\}$ be a proper, lower semicontinuous, convex function with $\dom f$ unbounded. Then
\[
    f^\circ(\infty;d)\le f^\infty(d)\qquad\text{for every } d\in\R^n.
\]
\end{proposition}

\begin{proof}
Fix $d\in\dom f^\infty$. For every $x\in\dom f$ and every $t>0$, it follows from \cite[Theorem~8.5]{rockafellar-convex} that 
\[
	f(x+td) - f(x) \leq tf^\infty(d).
\]
Dividing by $t$, and taking the limsup give us $f^\circ(\infty;d)\le f^\infty(d)$. For $d\notin\dom f^\infty$, $f^\infty(d)=+\infty$ and the inequality is trivial.
\end{proof}

We close this subsection with two computations and a counterexample showing that the equalities above genuinely require Lipschitzness at infinity.

\begin{example}
	Given a sublinear function $f:\R^n\to\R$. Then, $f$ is globally Lipschitz, so Theorem~\ref{prop: geometry of subdifferential} and Corollary~\ref{rem: classical subdiff at infinity} give us
\[
	f^\circ(\infty;\cdot)=\sigma_{\partial f(\infty)}=\sigma_{\partial f(0)}=f.
\]
Hence, the directional derivative at infinity of a sublinear function recovers itself.
\end{example}

\begin{example}
	Given a nonempty unbounded polyhedral convex set $C\subset\R^n$, define $f:=\delta_C$. On the one hand, Theorem~\ref{prop: polyhedral cones at infinity} makes $T(\infty;C)$ explicit, and a direct computation gives us
\[
    f^\circ(\infty;\cdot) = \delta_{T(\infty;C)}.
\]	
On the other hand, by Example~\ref{ex: indicator}, we have $\partial f(\infty)=\partial^\infty f(\infty)=N(\infty;C)$. It then follows from Proposition~\ref{prop: clarke deriv support} that $\Phi_f(d)=\sigma_{N(\infty;C)}(d)=0$ for $d \in T(\infty;C)$ and $+\infty$ otherwise. Hence,
\[
	\Phi_f=\delta_{T(\infty;C)}=f^\circ(\infty;\cdot).
\]
\end{example}

\begin{example} \label{ex: exp counterexample}
	On $\R$, let $f(x):=x^2$, which is not Lipschitz at infinity. A direct computation gives us
\[
    T(\infty_I;\epi f) = \{0\}\times\R, \qquad f^\circ(\infty;d) = \begin{cases}
    0		&	d = 0, \\
    +\infty	&	d \neq 0.
\end{cases}
\]
Hence $\Phi_f(0)=\inf\{\alpha\mid (0,\alpha)\in T(\infty_I;\epi f)\}=-\infty$, so $\Phi_f(0)<f^\circ(\infty;0)$. Thus the inequality in Proposition~\ref{prop: clarke deriv epi}(i) is strict, and Lipschitzness at infinity cannot be dropped from part~(ii).
\end{example}

\section{A Multiplier Rule for Convex Programs at Infinity} \label{sec: optimality at infinity}

In this section we illustrate how the convex calculus at infinity of Sections~\ref{sec: tangent and normal cones at infinity} and~\ref{sec: subdifferential at infinity} yields a multiplier-type rule for convex programs over unbounded feasible sets, decomposing asymptotic stationarity into an objective part and a feasible-set part through the sum rule at infinity. Throughout, $f:\R^n\to\R\cup\{\infty\}$ is a proper, lower semicontinuous, convex function with $\dom f$ unbounded, and $C\subset\R^n$ is a nonempty closed convex set such that $C\cap\dom f$ is unbounded. We consider the convex program
\begin{equation} \label{eq: convex program}
    \inf\{f(x)\mid x\in C\}.
\end{equation}
Writing $\bar f := f + \delta_C$, the program~\eqref{eq: convex program} is equivalent to minimizing $\bar f$ over $\R^n$.

\subsection{A Multiplier Rule at Infinity}

By Example~\ref{ex: indicator}, we have $\partial\delta_C(\infty)=\partial^\infty\delta_C(\infty)=N(\infty;C)$. Applying the sum rule of Proposition~\ref{prop: sum rule same-variable} to $f$ and $g:=\delta_C$ yields the following multiplier-type rule.

\begin{theorem} \label{thm: multiplier rule at infinity}
Under the asymptotic normal-separation condition
\begin{equation} \label{eq: qualification convex programming}
    \partial^\infty f(\infty) \cap \bigl(-N(\infty;C)\bigr) = \{0\},
\end{equation}
we have
\[
    \partial(f+\delta_C)(\infty) \subset \partial f(\infty) + N(\infty;C).
\]
In particular, if $0\in\partial(f+\delta_C)(\infty)$, then there exist
\[
    u\in\partial f(\infty),\qquad v\in N(\infty;C), \qquad u+v=0.
\]
\end{theorem}

\begin{proof}
The qualification \eqref{eq: singular separation} of Proposition~\ref{prop: sum rule same-variable} reduces to \eqref{eq: qualification convex programming} since $\partial^\infty \delta_C(\infty)=N(\infty;C)$. The conclusion follows since $\partial \delta_C(\infty)=N(\infty;C)$ as well.
\end{proof}

Theorem~\ref{thm: multiplier rule at infinity} should be understood as a \emph{decomposition result for asymptotic stationarity}: once $0\in\partial(f+\delta_C)(\infty)$ is present, the theorem separates the objective and constraint contributions at infinity into a subgradient of $f$ at infinity and a normal direction of $C$ at infinity. The condition~\eqref{eq: qualification convex programming} is a normal-separation condition used to justify the sum rule; it is not itself an attainment criterion, and the theorem does not assert that such stationarity must hold in every nonattainment situation. The condition~\eqref{eq: qualification convex programming} is the natural asymptotic counterpart of the basic constraint qualification $\partial^\infty f(\bar x)\cap(-N(\bar x;C))=\{0\}$ used in convex programming at finite points; see, e.g., \cite[Chapter~4]{nambook}.

Although Theorem~\ref{thm: multiplier rule at infinity} takes asymptotic stationarity as input, its contrapositive yields an attainment criterion: if $0\notin\partial(f+\delta_C)(\infty)$, then the convex program~\eqref{eq: convex program} attains its infimum, provided the latter is finite.

\begin{theorem} \label{thm: attainment infinity}
Under the standing assumptions of this section, assume $\inf_{x\in C}f(x)>-\infty$ and
\[
    0\notin \partial(f+\delta_C)(\infty).
\]
Then $\inf_{x\in C}f(x)$ is attained, i.e., there exists $x^*\in C\cap\dom f$ with $f(x^*)=\inf_{x\in C}f(x)$.
\end{theorem}

\begin{proof}
Set $\bar f := f+\delta_C$ and $\bar\alpha:=\inf\bar f\in\R$. The function $\bar f$ is proper, lower semicontinuous, convex with $\bar\alpha>-\infty$; in particular, $\dom\bar f=C\cap\dom f$ is nonempty, so fix any $y_0\in C\cap\dom f$, for which $\bar f(y_0)<\infty$. For each $\lambda>0$, the proximal mapping centered at $y_0$,
\[
    y_\lambda := \arg\min_{y\in\R^n}\Bigl\{\bar f(y) + \tfrac{1}{2\lambda}\|y-y_0\|^2\Bigr\},
\]
is well-defined and single-valued, and the first-order condition gives us
\[
    u_\lambda := -\tfrac{y_\lambda-y_0}{\lambda} \in \partial\bar f(y_\lambda).
\]
Taking $y=y_0$ in the defining infimum yields $\bar f(y_\lambda) + \tfrac{1}{2\lambda}\|y_\lambda-y_0\|^2 \le \bar f(y_0)$, hence
\[
    \tfrac{1}{2\lambda}\|y_\lambda-y_0\|^2 \le \bar f(y_0)-\bar f(y_\lambda) \le \bar f(y_0)-\bar\alpha,
\]
so $\|u_\lambda\|^2 = \|y_\lambda-y_0\|^2/\lambda^2 \le 2(\bar f(y_0)-\bar\alpha)/\lambda \to 0$ as $\lambda\to\infty$. Moreover, for any minimizing sequence $(y_k^*)$ with $\bar f(y_k^*)\to\bar\alpha$, the bound $\bar f(y_\lambda)\le \bar f(y_k^*) + \tfrac{1}{2\lambda}\|y_k^*-y_0\|^2$ gives us $\limsup_{\lambda\to\infty}\bar f(y_\lambda)\le\bar\alpha$, while $\bar f(y_\lambda)\ge\bar\alpha$; hence $\bar f(y_\lambda)\to\bar\alpha$.

Suppose to the contrary that $\bar\alpha$ is not attained. Lower semicontinuity of $\bar f$ rules out bounded subsequences of $(y_\lambda)$ — any cluster point would be a minimizer. Hence $\|y_\lambda\|\to\infty$ as $\lambda\to\infty$. Since $u_\lambda\in\partial\bar f(y_\lambda)$ gives us $(u_\lambda,-1)\in N\big((y_\lambda,\bar f(y_\lambda));\epi\bar f\big)$, and $\|\pi_I(y_\lambda,\bar f(y_\lambda))\|=\|y_\lambda\|\to\infty$ with $(u_\lambda,-1)\to(0,-1)$, the cone closedness of Proposition~\ref{prop: basic properties}(iii) applied to $\epi\bar f$ yields $(0,-1)\in N(\infty_I;\epi\bar f)$, i.e., $0\in\partial\bar f(\infty)$, contradicting the hypothesis.
\end{proof}

Theorem~\ref{thm: attainment infinity} complements the recession-cone criteria for attainment in classical convex analysis (e.g.,~\cite[Theorem~27.3]{rockafellar-convex}): the absence of asymptotic stationarity, captured by the convex set $\partial(f+\delta_C)(\infty)$, suffices for attainment. The condition $0\notin\partial(f+\delta_C)(\infty)$ is verifiable from the explicit calculus of Section~4 in many cases, e.g.\ via the sum rule of Proposition~\ref{prop: sum rule same-variable} when the qualification~\eqref{eq: qualification convex programming} holds.

The next proposition provides a sufficient condition for the qualification~\eqref{eq: qualification convex programming}, formulated entirely in terms of recession data of $f$ and $C$. It can be viewed as an asymptotic analog of the classical Slater regularity condition.

\begin{proposition} \label{thm: asymptotic Slater}
Assume that
\begin{equation}\label{eq: asymptotic Slater}
    \inte(C^\infty) \cap \dom f^\infty \neq \emptyset.
\end{equation}
Then the qualification~\eqref{eq: qualification convex programming} holds.
\end{proposition}

\begin{proof}
Let $d_0\in\inte(C^\infty)\cap\dom f^\infty$. Take any $u\in\partial^\infty f(\infty)\cap(-N(\infty;C))$; we show that $u=0$.

By Proposition~\ref{prop: geometry of singular}, $u\in(\dom f^\infty)^\circ$, so $\langle u,d_0\rangle\le 0$ since $d_0\in\dom f^\infty$.

Since $-u\in N(\infty;C)=(T(\infty;C))^\circ$ and $C^\infty\subset T(\infty;C)$ by Proposition~\ref{prop: recession inclusion}, we have $\langle u,w\rangle\ge 0$ for every $w\in C^\infty$.

Since $d_0\in\inte(C^\infty)$, there exists $\varepsilon>0$ such that $d_0+\delta\in C^\infty$ for every $\delta$ with $\|\delta\|<\varepsilon$. Hence $\langle u,d_0+\delta\rangle\ge 0$, equivalently,
\[
    \langle u,\delta\rangle\ge -\langle u,d_0\rangle\ge 0 \quad\text{for every } \|\delta\|<\varepsilon.
\]
If $u\neq 0$, the choice $\delta=-(\varepsilon/2)u/\|u\|$ yields $\langle u,\delta\rangle=-\varepsilon\|u\|/2<0$, contradicting the inequality above. Therefore $u=0$.
\end{proof}

\begin{remark}
The condition $\inte(C^\infty)\cap\dom f^\infty\neq\emptyset$ is only a sufficient condition for~\eqref{eq: qualification convex programming}. It is deliberately stronger than necessary and may fail for lower-dimensional recession cones $C^\infty$ even when the normal-separation condition~\eqref{eq: qualification convex programming} still holds.
\end{remark}

For convex programs presented through inequality constraints, the condition~\eqref{eq: asymptotic Slater} specializes to a Slater-type criterion at infinity.

\begin{corollary} \label{cor: asymptotic Slater g_i}
Under the standing assumptions of this section, let $C=\bigcap_{j=1}^m\{g_j\le 0\}$ for proper, lower semicontinuous, convex functions $g_j:\R^n\to\R\cup\{\infty\}$ with $\dom g_j^\infty=\R^n$. Suppose there exists $d_0\in\dom f^\infty$ such that
\[
    g_j^\infty(d_0)<0 \quad \text{for every } j=1,\ldots,m.
\]
Then the qualification~\eqref{eq: qualification convex programming} holds, and consequently the multiplier rule of Theorem~\ref{thm: multiplier rule at infinity} applies.
\end{corollary}

\begin{proof}
Each $g_j^\infty$ is a proper, lower semicontinuous, sublinear function with $\dom g_j^\infty=\R^n$, hence finite-valued and continuous on $\R^n$. Continuity at $d_0$ together with $g_j^\infty(d_0)<0$ gives us $\varepsilon_j>0$ such that $g_j^\infty(d)<0$ for all $\|d-d_0\|<\varepsilon_j$. Setting $\varepsilon:=\min_j\varepsilon_j>0$, the ball $B(d_0;\varepsilon)$ lies in $\bigcap_{j=1}^m\{g_j^\infty<0\}\subseteq\bigcap_{j=1}^m\{g_j^\infty\le 0\}=C^\infty$, where the last equality uses the standard recession-cone formula for intersections of closed convex sublevel sets~\cite[Corollary~8.3.3]{rockafellar-convex} (applicable since $C\supseteq C\cap\dom f$ is nonempty by the standing assumption, so each sublevel set $\{g_j\le 0\}\supseteq C$ is nonempty). Hence $d_0\in\inte(C^\infty)\cap\dom f^\infty$, and Proposition~\ref{thm: asymptotic Slater} applies.
\end{proof}

Proposition~\ref{thm: asymptotic Slater} and Corollary~\ref{cor: asymptotic Slater g_i} parallel the classical Slater condition $\exists\,x_0\in\dom f$ with $g_j(x_0)<0$ for all $j$, which is the standard sufficient condition for the basic constraint qualification at every finite optimal point of a convex program; see, e.g., \cite[Chapter~4]{nambook}. The asymptotic Slater \eqref{eq: asymptotic Slater} plays the analogous role at infinity by ensuring the sum-rule qualification~\eqref{eq: qualification convex programming}.

\begin{corollary} \label{cor: KKT at infinity}
Assume that $C=\bigcap_{j=1}^m C_j$, where $C_j:=\{x\in\R^n\mid g_j(x)\le 0\}$ for proper, lower semicontinuous, convex functions $g_j:\R^n\to\R\cup\{\infty\}$. Assume further that an iterated version of the intersection rule in Proposition~\ref{prop: normal intersection rule} applies to the family $\{C_j\}_{j=1}^m$, so that
\[
    N(\infty;C)\subset \sum_{j=1}^m N(\infty;C_j).
\]
Then under~\eqref{eq: qualification convex programming}, asymptotic stationarity $0\in\partial(f+\delta_C)(\infty)$ admits the KKT-type decomposition
\[
    0=u+v_1+\cdots+v_m,
    \qquad
    u\in\partial f(\infty),\quad v_j\in N(\infty;C_j),\ j=1,\ldots,m.
\]
\end{corollary}

The two-set intersection rule in Proposition~\ref{prop: normal intersection rule} does not automatically extend to arbitrary finite families without further qualification conditions; the assumption above is therefore part of the hypothesis, not a consequence.

Corollary~\ref{cor: KKT at infinity} stops at the normal-cone level: it decomposes asymptotic stationarity into contributions from $\partial f(\infty)$ and the sublevel-set normal cones $N(\infty;C_j)$. The natural KKT closure would replace each $v_j\in N(\infty;C_j)$ by $v_j=\lambda_j a_j$ with $\lambda_j\ge 0$ and $a_j\in\partial g_j(\infty)$, yielding
\[
    0\in\partial f(\infty)+\sum_{j=1}^m\R_+\,\partial g_j(\infty).
\]
This step requires an asymptotic analog of the finite-point sublevel-set formula $N(x;\{g_j\le 0\})=\cone\partial g_j(x)$, which holds at active points under standard constraint qualifications. A naive extension to infinity already fails for simple sublinear constraints: for $g(x)=-x$ on $\R$, $\{g\le 0\}=\R_+$ gives us $N(\infty;\{g\le 0\})=\{0\}$, while $\cone\partial g(\infty)=\R_+(-1)=\R_-$. The precise asymptotic sublevel-set normal-cone formula---and hence the full KKT form---is delicate and is left for future work.

\subsection{Comparison with Classical Recession Criteria}

The qualification condition in Theorem~\ref{thm: multiplier rule at infinity} and the classical recession-cone criteria for attainment play different roles. The condition
\[
    \partial^\infty f(\infty)\cap\bigl(-N(\infty;C)\bigr)=\{0\}
\]
is a normal-separation condition at infinity: it guarantees that the subdifferential of $f+\delta_C$ at infinity can be estimated by a sum of the objective and constraint contributions. By contrast, classical attainment criteria are usually formulated in terms of recession directions, for example through $C^\infty$ and $f^\infty$, and are designed to rule out noncompact minimizing behavior; see, e.g., \cite[Theorem~27.3]{rockafellar-convex}. Therefore, Theorem~\ref{thm: multiplier rule at infinity} should not be viewed as a replacement for recession-cone attainment criteria. Rather, it provides a convex-analytic multiplier interpretation of asymptotic stationarity once such stationarity is known.

\begin{example}
Let $f(x):=\langle c,x\rangle$ on $\R^n$ for some $c\in\R^n$, and let $C\subset\R^n$ be a nonempty unbounded closed convex set. Then $f^\infty=f$, $\dom f^*=\{c\}$, $\partial f(\infty)=\{c\}$, and $\partial^\infty f(\infty)=\{0\}$, so~\eqref{eq: qualification convex programming} is trivially satisfied. Theorem~\ref{thm: multiplier rule at infinity} reads
\[
    \partial(f+\delta_C)(\infty)\subset \{c\}+N(\infty;C).
\]
Thus the asymptotic stationarity condition
\[
    0\in\partial(\langle c,\cdot\rangle+\delta_C)(\infty)
\]
implies
\[
    -c\in N(\infty;C).
\]
This identifies the negative objective direction as an asymptotic normal direction of the feasible set. This condition should be distinguished from the classical recession-cone boundedness condition $\langle c,d\rangle\ge 0$ for all $d\in C^\infty$, because in general $N(\infty;C)$ need not coincide with $(C^\infty)^\circ$. The example therefore illustrates the geometric meaning of the multiplier rule at infinity rather than replacing the standard recession analysis of linear programs.
\end{example}

The next example shows that the multiplier rule recovers a familiar dual-feasibility condition in a non-trivial convex program, namely the basis-pursuit problem of $\ell_1$-minimization.

\begin{example} \label{ex: l1 basis pursuit}
Let $f(x):=\|x\|_1$ on $\R^n$ and $C:=\{x\in\R^n\mid Ax=b\}$, where $A:\R^n\to\R^m$ is linear with $\ker A\neq\{0\}$ (so $C$ is an unbounded affine subspace) and $b\in\rge A$ (so $C$ is nonempty). The function $f$ is finite-valued sublinear with $\dom f^\infty=\R^n$, and by Theorem~\ref{prop: geometry of subdifferential},
\[
    \partial^\infty f(\infty)=\{0\},
    \qquad
    \partial f(\infty)=\dom f^*=\mathbb{B}_\infty,\qquad\text{where}\quad\mathbb{B}_\infty:=\bigl\{u\in\R^n\mid \|u\|_\infty\le 1\bigr\}.
\]
Since $C$ is an affine subspace, $T(x;C)=\ker A$ for every $x\in C$, hence
\[
    T(\infty;C)=\ker A,
    \qquad
    N(\infty;C)=(\ker A)^\circ=\rge(A^\top).
\]
The qualification condition~\eqref{eq: qualification convex programming} is trivially satisfied since $\partial^\infty f(\infty)=\{0\}$, and Theorem~\ref{thm: multiplier rule at infinity} yields
\[
    \partial(f+\delta_C)(\infty)\subset \mathbb{B}_\infty+\rge(A^\top).
\]
The asymptotic stationarity condition $0\in\partial(f+\delta_C)(\infty)$ thus implies the existence of $\lambda\in\R^m$ with
\[
    -A^\top\lambda\in\mathbb{B}_\infty,
    \qquad\text{equivalently,}\qquad
    \|A^\top\lambda\|_\infty\le 1.
\]
This is the standard dual-feasibility condition associated with the $\ell_1$ basis-pursuit program
\[
    \inf\bigl\{\|x\|_1\mid Ax=b\bigr\}.
\]
Thus, in this example, the multiplier rule at infinity recovers the dual-feasibility geometry associated with the affine constraint. It should not be interpreted as a full KKT optimality condition for a finite minimizer; rather, it shows how asymptotic stationarity, when present, decomposes into the polar geometry of the $\ell^1$-unit ball and the normal geometry of the affine feasible set.
\end{example}

\begin{remark} \label{rem: necessary perspective}
Theorem~\ref{thm: multiplier rule at infinity} is a decomposition statement that takes asymptotic stationarity as input. Read in the contrapositive, it provides a check for the absence of asymptotic stationarity: if $u+v\neq 0$ for every $u\in\partial f(\infty)$ and every $v\in N(\infty;C)$, then $0\notin\partial(f+\delta_C)(\infty)$, indicating that no asymptotic stationarity occurs at infinity. Such a check is informative when $\partial f(\infty)$ and $N(\infty;C)$ admit explicit descriptions, as in the sublinear and the affine-subspace regimes treated above. A systematic analysis of when $\partial(f+\delta_C)(\infty)$ actually contains $0$, in terms of recession-cone descent or other classical asymptotic conditions, is beyond the scope of this paper and is left for future work.
\end{remark}

\section{Conclusion}
We have developed a convex-analytic theory of tangent cones, normal cones, and subdifferentials at infinity. For convex sets, we obtained tangent--normal polarity at infinity, recession-based characterizations with a checkable polyhedral criterion for the equality $T(\infty_I;\Omega)=\Omega^\infty$, and product, intersection, and linear-image calculus rules. For convex functions, the epigraph-based construction yields closed convex subdifferentials at infinity, sandwich inclusions involving the Fenchel conjugate and the recession function, and sum, chain, max, infimal-convolution, marginal-function, and Clarke-type directional-derivative rules. As an application, we obtained a multiplier-type decomposition of asymptotic stationarity and a corresponding attainment criterion for convex programs over unbounded feasible sets.

Three directions seem natural for further work:
\begin{itemize}
    \item \emph{Directional theory at infinity.} Developing directional subdifferentials, directional normal cones, and refined calculus rules in the convex setting. The Clarke-type directional derivative of Section~\ref{subsec: clarke deriv infinity} is a first step.
    \item \emph{Exact calculus for infimal convolution and marginal/value functions.} Sharpening the inclusions obtained here into exact formulas under verifiable qualification conditions, including an asymptotic analog of the sublevel-set normal-cone formula needed to close the KKT-type decomposition of Corollary~\ref{cor: KKT at infinity}.
    \item \emph{Recession-function and perturbation/conjugation connections.} Linking the convex-analytic objects at infinity to asymptotic functions, conjugate duality, and perturbation theory in the spirit of classical convex analysis.
\end{itemize}

\medskip
{\noindent \textbf{Acknowledgments.}
   This work was supported by National Foundation for Science and Technology Development of Vietnam (NAFOSTED) grant number 101.01-2025.14.
}

\medskip
{\noindent \textbf{Declarations}
\begin{itemize}
	\item Competing interests: The authors have no relevant financial or non-financial interests to disclose.
	\item Data Availability: Data sharing is not applicable to this article as no datasets were generated or analyzed during the current study.
	\item Author contributions: The authors contributed equally to this work.
\end{itemize}
} 

\bibliographystyle{mystyle}
\bibliography{References}

@book{aubin,
	year = {1990},
	publisher = {Birkh\"{a}user Boston},
	author = {Jean-Pierre Aubin and H{\'{e}}l{\`{e}}ne Frankowska},
	title = {Set-valued analysis}
}

@book{jahnbook,
	year = {2004},
	publisher = {Springer Berlin-Heidelberg},
	author = {Johannes Jahn},
	title = {Vector Optimization}
}

@book{nambook,
  author = {Boris S. Mordukhovich and Nguyen Mau Nam},
  title = {An Easy Path to Convex Analysis and Applications},
  publisher = {Springer International Publishing},
  year = {2014}
}

@book{rockafellar-convex,
  author    = {R. Tyrrell Rockafellar},
  title     = {Convex Analysis},
  publisher = {Princeton University Press},
  year      = {1970}
}

@article{tung23a,
	author = {Nguyen, Minh Tung and Pham, Tien-Son},
	title = {Clarke's Tangent Cones, Subgradients, Optimality Conditions, and the Lipschitzness at Infinity},
	journal = {SIAM J. Optim.},
	volume = {34},
	number = {2},
	pages = {1732-1754},
	year = {2024}
}

@article{tung23b,
	title = {Subdifferentials at infinity and applications in optimization},
	volume = {214},
	number = {1–2},
	journal = {Math. Program.},
	publisher = {Springer Science and Business Media LLC},
	author = {Kim,  Do Sang and Nguyen, Minh Tung and Pham,  Tien-Son},
	year = {2025},
	pages = {409–440}
}

@book{rockafellar-wets,
  author    = {R. Tyrrell Rockafellar and Roger J.-B. Wets},
  title     = {Variational Analysis},
  series    = {Grundlehren der Mathematischen Wissenschaften},
  volume    = {317},
  publisher = {Springer},
  year      = {1998},
}

@book{mordukhovich-vol1,
  author    = {Boris S. Mordukhovich},
  title     = {Variational Analysis and Generalized Differentiation I: Basic Theory},
  series    = {Grundlehren der Mathematischen Wissenschaften},
  volume    = {330},
  publisher = {Springer},
  year      = {2006}
}

@article{tuyen-bae-kim-minimax,
  author       = {Nguyen Van Tuyen and Kwan Deok Bae and Do Sang Kim},
  title        = {Optimality Conditions at Infinity for Nonsmooth Minimax Programming Problems with Some Applications},
  journal      = {Journal of Optimization Theory and Applications},
  year         = {2025},
  volume       = {205},
  number       = {2},
  articlenumber= {32},
  doi          = {10.1007/s10957-025-02652-1}
}

@misc{kim-pham-tung-tuyen-coderivative,
  author       = {Do Sang Kim and Tien-Son Pham and Nguyen Minh Tung and Nguyen Van Tuyen},
  title        = {Coderivatives at Infinity of Set-Valued Mappings with Applications to Optimization},
  year         = {2025},
  note         = {Preprint}
}

@misc{kien-tuyen-nghi-directional,
  author       = {Le Ngoc Kien and Nguyen Van Tuyen and Tran Van Nghi},
  title        = {Directional Subdifferentials at Infinity and Its Applications},
  year         = {2025},
  eprint       = {2510.09179},
  archivePrefix= {arXiv},
  primaryClass = {math.OC}
}

@article{anh-hung-relative,
  author       = {Nguyen Le Hoang Anh and Nguyen Canh Hung},
  title        = {Normal Cone and Subdifferential with Respect to a Set at Infinity and Their Applications},
  journal      = {Journal of Optimization Theory and Applications},
  year         = {2025},
  volume       = {206},
  number       = {3},
  articlenumber= {77},
  doi          = {10.1007/s10957-025-02761-x}
}

@article{anh-hung-jgo,
  author       = {Nguyen Canh Hung and Nguyen Le Hoang Anh},
  title        = {Optimality Conditions and Solution Analysis at Infinity for Nonsmooth Optimization Problems},
  journal      = {Journal of Global Optimization},
  year         = {2026},
  volume       = {94},
  number       = {2},
  pages        = {569--591},
  doi          = {10.1007/s10898-026-01590-0}
}

@article{moreau1962,
  author  = {Moreau, Jean-Jacques},
  title   = {D{\'e}composition orthogonale d'un espace hilbertien selon deux c{\^o}nes mutuellement polaires},
  journal = {Comptes Rendus de l'Acad{\'e}mie des Sciences de Paris},
  volume  = {255},
  year    = {1962},
  pages   = {238--240}
}
\end{document}